\documentclass[final]{amsart}
\pdfoutput=1
\usepackage{amssymb,amscd}
\usepackage[all]{xy}
\usepackage{graphicx}
\usepackage{mathrsfs}
\usepackage{dsfont}
\usepackage{CJK}
\usepackage[resetfonts]{cmap}

\newcommand{\Z}{\mathbb Z}

 \theoremstyle{plain}
\newtheorem{theorem}{Theorem}[section]
\newtheorem{corollary}[theorem]{Corollary}
\newtheorem{lemma}[theorem]{Lemma}
\newtheorem{proposition}[theorem]{Proposition}
\newtheorem{definition-proposition}[theorem]{Definition/Proposition}
\newtheorem{construction}[theorem]{Construction}
\newtheorem{axiom}[theorem]{Axiom}

 \theoremstyle{definition}
\newtheorem{definition}[theorem]{Definition}

\newtheorem{example}[theorem]{Example}

\theoremstyle{remark}

\newtheorem{remark}[theorem]{Remark}

\numberwithin{equation}{section}

\usepackage{float}
\usepackage{hyperref}
\usepackage{geometry}
\begin{document}

\title{A Generalization of Anti-homomorphisms}

\author{Tianwei Liang}


\maketitle
\begin{abstract}
We prove some nice properties of anti-homomorphisms, some of which are analogic to that of homomorphisms. Meanwhile, we develop a new kind of composition called $*$-composition such that the $*$-composition of two anti-homomorphisms is still an anti-homomorphism. Moreover, we develop a certain kind of categories called factorization categories, which generalize anti-homomorphisms and provide a general framework to study ``anti'' phenomenon. We deduce several results of anti-homomorphisms in great generality.
\end{abstract}

\tableofcontents

\section{Introduction}
\subsection{Background and motivation}\

It is known that a \textit{homomorphism} $f:G\rightarrow H$ of groups (written multiplicatively) preserves products, i.e. $f(xy)=f(x)f(y)$ for $x,y\in G$. However, an \textit{anti-homomorphisms} $f^{*}:G\rightarrow H$ of groups (written multiplicatively) reverses products, i.e. $f(xy)=f(y)f(x)$ for $x,y\in G$. The literatures of anti-homomorphisms of groups, rings, etc. are scarce. Since, for example, the composition of two anti-homomorphisms is a homomorphism rather than an anti-homomorphism, such topic is not commonly studied by many. In addition, being anti-isomorphic can not form an equivalence relation on groups, and we can not form a category whose objects are groups and morphisms are anti-homomorphisms. It seems that anti-homomorphism does not capture some nice properties.

However, some authors like C. C. MacDuffee have touched this concept in a lighter way. And, in 2006, Chandrasekhara Rao, K. and G. Gopalakrishnamoorthy, \cite{Gopalakrishnamoorthy}, defined a new composition of anti-homomorphisms so that such composition of two anti-homomorphisms is still an anti-homomorphism. And they also prove analogue of the basic results of group theory on homomorphisms for anti-homomorphisms, including the isomorphism theorems. So our main idea is to transport some nice properties of homomorphisms to anti-homomorphisms.

\subsection{The notions of ``anti''}\

We talk more about the notions of ``\textit{anti}''. We speak of a \textit{homomorphism} $\varphi : A \rightarrow B$ of rings a mapping $\varphi:A\rightarrow B$ such that $\varphi(x+y)=\varphi(x)+\varphi(y)$ and $\varphi(xy)=\varphi(x)\varphi(y)$ for $x,y\in A$. When we speak of an \textit{anti-homomorphism} $\varphi : A \rightarrow B$ of rings, we mean a mapping $\varphi : A \rightarrow B$ such that $\varphi(x+y) = \varphi(x) + \varphi(y)$ and $\varphi(xy)=\varphi(y)\varphi(x)$ for $x,y\in A$. Moreover, an \textit{anti-linear mapping} $\varphi : V \rightarrow W$ of complex vector spaces is a mapping $\varphi:V\rightarrow W$ such that $\varphi(x+y)=\varphi(x)+\varphi(y)$ and $\varphi(\lambda x) =\overline{\lambda} \varphi(x)$ for $x,y\in V$ and $\lambda\in\mathbb{C}$. It's natural to interpret the concept ``anti'' as a kind of duality.

\subsection{Further results obtained in the paper}\

In this paper, we will study the nature of the so-called anti-homomorphisms and deduce some nice properties of anti-homomorphisms that homomorphisms enjoy. Then we generalize anti-homo\-morphisms to the so-called anti-morphisms and develop a certain kind of categories called factorization categories. Using these notions, we deduce some general results of anti-morphisms that morphisms hold in the original category. Let us sketch our main results in the following.

\subsection{Main results in the paper}\

In Section \ref{B1}, we investigate the category $\textbf{Groups}$ of groups. We first lay down some basic knowledge about anti-homomorphisms of groups.

An important result is the following, due to K. Chandrasekhara Rao and Swaminathan Venkataraman in \cite[Theorem 2.3, Theorem 2.4]{Chandrasekhara}.
\begin{proposition}
Let $\varphi:A\rightarrow B$ be a homomorphism (resp. an anti-homomorphism) of groups and $\psi:B\rightarrow C$ be an anti-homomorphism (resp. a homomorphism) of groups. Then $\psi\circ\varphi:A\rightarrow C$ is an anti-homomorphism of groups. If $\varphi:A\rightarrow B$ and $\psi:B\rightarrow C$ are anti-homomorphism of groups, then $\psi\circ\varphi:A\rightarrow C$ is a homomorphism of groups.
\end{proposition}

This inspires us to think conversely that if any homomorphism of groups can be factored into a composition of anti-homomorphisms of groups. Then an observation that every homomorphism can be expressed as a composition of a reverse mapping and an anti-homomorphism proves the following theorem:
\begin{theorem}
Every homomorphism of groups is expressible as a composite of anti-homomorphisms of groups.
\end{theorem}

To offset the flaw that being anti-isomorphic can not form an equivalence relation on groups. We define an equivalence relation of anti-homomorphisms so that two pairs of anti-homomorphisms are \textit{equivalent} if their composition is the same homomorphism. The equivalence relation enables us to define a function that assigns to each homomorphism its equivalence class consisting of pairs of anti-homomorphisms. We call this function \textit{law of factorization}.
\begin{proposition}
To each triple of objects $A,B,C\in {\rm{Ob}}(\textbf{Groups})$, we have a law of factorization
$$
  {\rm{Hom}}(A,C) \rightarrow [An(A,B),An(B,C)], \ \ f\mapsto [f].
$$
\end{proposition}

Due to the lost literature, \cite{Gopalakrishnamoorthy}, we reprove the anti-factorization theorem, the anti-homomorphism theorem, the second anti-isomorphism theorem and the third anti-isomorphism theorem, which are analogue of the factorization theorem, the homomorphism theorem, the second isomorphism theorem, and the third isomorphism theorem respectively.

\begin{theorem}[Anti-factorization Theorem]
Let $N$ be a normal subgroup of $G$ and $\varphi:G\rightarrow H$ an anti-homomorphism whose kernel contains $N$, then there exists a unique anti-homomorphism $\psi:G/N\rightarrow H$ such that the following diagram commutes ($\pi$ is the natural projection):
$$
\xymatrix{
  G \ar[dr]_{\varphi} \ar[r]^{\pi}
                & G/N \ar@{-->}[d]^{\exists!\psi}  \\
                & H             }
$$
\end{theorem}

\begin{theorem}[The Anti-homomorphism Theorem]
If $\varphi:A\rightarrow B$ is an anti-homomorphism, then there is a canonical anti-isomorphism
$$
A/{\rm{Ker}}\varphi\cong {\rm{Im}}\varphi,
$$
that is, there is a unique anti-isomorphism $\xi:A/{\rm{Ker}}\varphi\xrightarrow{\sim} {\rm{Im}}\varphi$ such that the following diagram commutes:
$$
\xymatrix{
  A \ar[d]_{\pi} \ar[r]^{\varphi} & B  \\
  A/{\rm{Ker }}\varphi \ar@{-->}[r]^{\exists!\xi} & {\rm{Im }}\varphi \ar[u]^{\imath}  }
$$
($\imath$ is the inclusion homomorphism and $\pi$ is the natural projection.)
\end{theorem}

\begin{theorem}[The Second Anti-isomorphism Theorem]
Let $A$ be a group and $B,C$ be normal subgroups of $A$. If $C\subset B$, then $C$ is a normal subgroup of $B$ and $B/C$ is a normal subgroup of $A/C$. Then there exists a unique anti-isomorphism
$$
A/B\cong(A/C)/(B/C),
$$
that is, there exists a unique anti-isomorphism $\xi:A/B\xrightarrow{\sim}(A/C)/(B/C)$ such that the following diagram commutes:
$$
\xymatrix{
  A \ar[d]_{\rho^{*}} \ar[rr]^{\pi} &    & A/C \ar[d]^{\tau}\ar[lld]_{\sigma} \\
  A/B \ar@{-->}[rr]^{\exists!\xi} &   & (A/C)/(B/C)   }
$$
($\pi,\tau$ are natural projections and $\rho^{*}$ is an anti-projection, i.e. for the natural projection $\rho$, $\rho^{*}=\rho\circ1^{*}$).
\end{theorem}

\begin{theorem}[The Third Anti-isomorphism Theorem]
Let $A$ be a subgroup of a group $G$, and let $N$ be a normal subgroup of $G$. Then $AN$ is a subgroup of $G$, $N$ is a normal subgroup of $AN$, $A\cap N$ is a normal subgroup of $A$, and there exists a unique anti-isomorphism
$$
AN/N\cong A/(A\cap N),
$$
that is, there is a unique anti-isomorphism $\xi:AN/N\xrightarrow{\sim} A/(A\cap N)$ such that the following diagram commutes:
$$
\xymatrix{
  A \ar[d]_{\imath}\ar[rrd]^{\varphi} \ar[rr]^{\rho} &   & A/(A\cap N) \ar@{-->}[d]^{\exists!\xi} \\
  AN \ar[rr]^{\pi^{*}} &  & AN/N   }
$$
($\imath$ is the inclusion homomorphism, $\rho$ is the natural projection, and $\pi^{*}$ is the anti-projection).
\end{theorem}

Then we define a kind of composition called \textit{$*$-composition} such that the $*$-composition of two anti-homomorphisms is still an anti-homomorphism. Using $*$-composition, we can finally form the category of groups and anti-homomorphisms, which is denoted by $\textbf{Groups-An}$. It is observed that one can define a function that assigns to each homomorphism an anti-homomorphism. In fact, we have the following proposition:
\begin{proposition}
There is a one-to-one correspondence of sets ${\rm{Hom}}(A,B)\xrightarrow{\sim} An(A,B)$ for any $A,B\in {\rm{Ob}}(\textbf{Groups})$.
\end{proposition}

Using the above proposition, we arrive at the following theorem:
\begin{theorem}
The category $\textbf{Groups}$ of groups and homomorphism and the category $\textbf{Groups-An}$ of groups and anti-homomorphisms are equivalent.
\end{theorem}

In Section \ref{B2}, we investigate the category $\textbf{Rings}$ of rings. We also first lay down some basic knowledge of anti-homomorphisms of rings. Then we deduce several nice results, some of which are more or less the same as the case of groups. Next, we list out some interesting results.

We show that an anti-homomorphism of rings preserves subrings and may reverse ideals.
\begin{proposition}
Let $\varphi:A\rightarrow B$ be an anti-homomorphism of rings. If $A_{0}$ is a subring of $A$, then
$$
\varphi(A_{0})=\{\varphi(x)\mid x\in A_{0}\}
$$
is a subring of $B$. If $B_{0}$ is a subring of $B$, then
$$
\varphi^{-1}(B_{0})=\{x\in A\mid \varphi(x)\in B_{0}\}
$$
is a subring of $A$. Moreover, if $\mathfrak{a}$ is a left-ideal of $A$ and $\varphi$ is an anti-epimorphism, then $\varphi(\mathfrak{a})$ is a right-ideal of $B$. And if $\mathfrak{b}$ is a left-ideal of $B$ and $\varphi$ is an anti-epimorphism, then $\varphi^{-1}(\mathfrak{b})$ is a right-ideal of $A$.
\end{proposition}

We also show that every homomorphism can be expressed as a composition of anti-homomorphisms.
\begin{theorem}
Every homomorphism of rings is expressible as a composite of anti-homomorphisms of rings.
\end{theorem}

And we prove anti-factorization theorem and anti-homomorphism theorem in the case of rings.
\begin{theorem}[Anti-factorization Theorem]
Let $I$ be an ideal of a ring $R$. For every anti-homomorphism $f:R\rightarrow S$ of rings whose kernel contains $I$, there exists a unique anti-homomorphism $\psi:R/I\rightarrow S$ such that the following diagram commutes ($\pi$ is the natural projection):
$$
\xymatrix{
  R \ar[dr]_{f} \ar[r]^{\pi}
                & R/I \ar@{-->}[d]^{\psi}  \\
                & S             }
$$
\end{theorem}

\begin{theorem}[Anti-homomorphism Theorem]
If $\varphi:R\rightarrow S$ is an anti-homomorphism of rings, then there is a canonical anti-isomorphism
$$
R/{\rm{Ker}}\varphi\cong {\rm{Im}}\varphi,
$$
that is, there is a unique anti-isomorphism $\xi:R/{\rm{Ker }}\varphi\xrightarrow{\sim} {\rm{Im }}\varphi$ such that the following diagram commutes:
$$
\xymatrix{
  R \ar[d]_{\pi} \ar[r]^{\varphi} & S  \\
  R/{\rm{Ker}}\varphi \ar@{-->}[r]^{\xi} & {\rm{Im}}\varphi \ar[u]^{\imath}  }
$$
($\pi$ is the natural projection and $\imath$ is the inclusion homomorphism).
\end{theorem}

The following are the results that we want to put into a larger context:
\begin{definition-proposition}
Consider the category $\mathcal{C}$ which is \textbf{Groups} or \textbf{Rings}.
\begin{itemize}
  \item [(1)]
  To each pair of objects \textit{A,B} $\in {\rm{Ob}}(\mathcal{C})$, we denote the {\rm{set of homomorphisms}} by ${\rm{Hom}}(A,B)$ and the {\rm{set of anti-homomorphisms}} by $An(A,B)$. And we denote the {\rm{set of isomorphisms}} by ${\rm{Hom.Is}}(A,B)$ and the {\rm{set of anti-isomorphisms}} by $An{\rm{.Is}}(A,B)$.
  \item [(2)]
  To each triple of objects \textit{A,B,C} $\in {\rm{Ob}}(\mathcal{C})$, we have \textit{laws of compositions}
  \begin{align*}
  &An(A,B) \times An(B,C) \rightarrow {\rm{Hom}}(A,C)\\
  &An(A,B) \times {\rm{Hom}}(B,C) \rightarrow An(A,C)\\
  &{\rm{Hom}}(A,B) \times An(B,C) \rightarrow An(A,C)\\
  &{\rm{Hom}}(A,B) \times {\rm{Hom}}(B,C) \rightarrow {\rm{Hom}}(A,C)
  \end{align*}
  and \textit{a law of factorization}
  $$
  {\rm{Hom}}(A,C) \rightarrow [An(A,B),An(B,C)], \ \ f\mapsto [f].
  $$
  \item [(3)]
  For every $A,B\in {\rm{Ob}}(\mathcal{C})$, there exists a reverse mapping $1^{*}_{A}\in An(A,A)$ such that for any morphism $f\in {\rm{Hom}}(A,B)$, we have
  $f\circ 1^{*}_{A}\in An(A,B)$. And for any $g\in {\rm{Hom}}(B,A)$, we have $1^{*}_{A}\circ f\in An(B,A)$.
  \item [(4)]
  The compositions defined above are \textit{associative}.
\end{itemize}
\end{definition-proposition}

By comparing the ``anti'' phenomenon of $\textbf{Groups}$ and $\textbf{Rings}$, we think that such ``anti'' phenomenon can be studied in a more general categorical setup. This motivates us to define the so-called \textit{factorization categories} in Section \ref{B3}. Roughly speaking, a factorization category is a category $\mathscr{C}$ to which we adjoin anti-morphisms and satisfies some axioms.

Using factorization categories, we succeed in deducing some generalized results in $\textbf{Groups}$ and $\textbf{Rings}$. Since the composition of two anti-morphisms is a morphism, we also think conversely that if every morphism can be expressed as a composition of anti-morphisms. Then the same observation that every morphism can be written as a composition of a reverse morphism and an anti-morphism gives the following theorem:
\begin{theorem}
Every morphism can be expressible as a composite of anti-morphisms.
\end{theorem}

The above theorem yields an equivalence relation on anti-morphisms. Two pairs of composable anti-morphisms are \textit{equivalent} if their compositions is equal to the same morphism. The equivalence relation so defined derives a function that assigns to each morphism its equivalence class of pairs of anti-morphisms. We call this function \textit{law of factorization}.
\begin{proposition}
Let $\mathscr{C}$ be a factorization category. For any $A,B,C\in {\rm{Ob}}(\mathscr{C})$, we have a law of factorization
$$
{\rm{Hom}}(A,C) \rightarrow [An(A,B),An(B,C)], \ \ f\mapsto [f].
$$
\end{proposition}

We find that factorization category may inherit some properties from morphisms in the underlying category to anti-morphisms. This enables us to define the \textit{kernels}, \textit{cokernels}, \textit{coimages}, and \textit{images} of anti-morphisms, as well as \textit{anti-monomorphisms} and \textit{anti-epimorphisms}. Then using these notions, we prove several results of anti-morphisms that morphisms hold in a preadditive category and an abelian category.

\begin{lemma}
Let $f^{*}:x\rightarrow y$ be an anti-morphism in a preadditive factorization category. If the kernel, cokernel, coimage, and image exist, then $f^{*}$ factors uniquely as the following sequence:
$$
x\rightarrow {\rm{Coim}}(f^{*})\rightarrow {\rm{Im}}(f^{*})\rightarrow y.
$$
\end{lemma}

\begin{lemma}
Let $f^{*}:x\rightarrow y$ be an anti-morphism in an abelian factorization category $\mathscr{C}$. Then
\begin{itemize}
  \item [(1)]
  $f^{*}$ is injective if and only if $f^{*}$ is an anti-monomorphism.
  \item [(2)]
  $f^{*}$ is surjective if and only if $f^{*}$ is an anti-epimorphism.
\end{itemize}
\end{lemma}

More importantly, we deduce the generalized anti-homomorphism theorem, the generalized anti-factorization theorem, and the generalized second anti-isomorphism theorem in the case of abelian factorization category.

\begin{theorem}[Generalized anti-homomorphism theorem]
Let $\mathscr{C}$ be an abelian factorization category and let $f^{*}:x\rightarrow y$ be an anti-morphism in $\mathscr{C}$. Then we have an anti-isomorphism $x/{\rm{Ker}}(f^{*})\cong{\rm{Im}}(f^{*})$. In particular, if $f^{*}$ is injective, then $x\cong{\rm{Im}}(f^{*})$. And if $f^{*}$ is surjective, then $x/{\rm{Ker}}(f^{*})\cong y$.
\end{theorem}

\begin{theorem}[Generalized anti-factorization theorem]
Let $\mathscr{C}$ be an abelian factorization category and let $f^{*}:x\rightarrow y$ be an anti-morphism in $\mathscr{C}$. If $\mu:z\rightarrow x$ is an injective morphism such that $f\mu=0$, then there exists a unique anti-morphism $\psi^{*}:x/z\rightarrow y$ such that the following diagram commutes ($\pi$ is the canonical map):
$$
\xymatrix{
  x \ar[dr]_{f^{*}} \ar[r]^{\pi}
                & x/z \ar@{-->}[d]^{\exists!\psi^{*}}  \\
                & y             }
$$
\end{theorem}

\begin{theorem}[Generalized second anti-isomorphism theorem]
Let $\mathscr{C}$ be an abelian factorization category and let $A,B,C\in{\rm{Ob}}(\mathscr{C})$ such that $C\subset B\subset A$. Then there is a unique anti-isomorphism
$$
(A/C)/(B/C)\cong A/B.
$$
\end{theorem}

Moreover, we can define \textit{$*$-composition} of anti-morphisms in a factorization category, so that the $*$-composition of two anti-morphisms is an anti-morphism. Using $*$-composition, we construct the \textit{anti-category} $\mathscr{C}^{An}$ \textit{associated to} $\mathscr{C}$ from a factorization category. It is a category whose objects are those of $\mathscr{C}$ and morphisms are anti-morphisms in $\mathscr{C}$. Note that one can define a function that assigns to each morphism an anti-morphism. In fact, we have the following proposition:
\begin{proposition}
Let $\mathscr{C}$ be any factorization category. For any $A,B\in {\rm{Ob}}(\mathscr{C})$, there is a one-to-one correspondence ${\rm{Hom}}(A,B)\xrightarrow{\sim} An(A,B)$ of sets.
\end{proposition}

Using the above proposition, we arrive at the following theorem:
\begin{theorem}
Let $\mathscr{C}$ be a factorization category. The associated anti-category $\mathscr{C}^{An}$ and the underlying category $\mathscr{C}$ are equivalent.
\end{theorem}

Hence, the above theorem readily yields the following theorem:
\begin{theorem}
If $\mathscr{C}$ is an abelian factorization category, then the associated anti-category $\mathscr{C}^{An}$ is an abelian category.
\end{theorem}

Moreover, by virtue of the $*$-composition, we can define the bifunctor $An(-,-)$ that to objects $A,B$ of the associated abelian anti-category $\mathscr{C}^{An}$, it associates the group of anti-morphisms from $A$ to $B$. It can be checked that $An(-,-)$ is isomorphic to the bifunctor $\textrm{Hom}(-,-)$.
\begin{theorem}
$An(-,-)$ is a bifunctor from the abelian associated anti-category $\mathscr{C}^{An}$ to the category $\textbf{Ab}$ of abelian groups. It is contravariant in the first variable, and covariant in the second variable. Moreover, the functors $An(-,-)$ and ${\rm{Hom}}(-,-)$ are isomorphic.
\end{theorem}

By investigating the universal property of products in a category, we prove two anti-analogues of universal property in a factorization category with products, which are called the \textit{anti-universal properties}.
\begin{proposition}Let $\mathscr{C}$ be a factorization category with products.

\begin{itemize}
  \item [(1)]
  Let $(X_{i})_{i\in I}$ be a family of objects of $\mathscr{C}$ indexed by a set $I$ and let $X=\prod_{i\in I}X_{i}$ be the product of $X_{i}$ with projections $p_{i}:X\rightarrow X_{i}$. Then given any object $Y\in{\rm{Ob}}(\mathscr{C})$ and anti-morphism $f_{i}^{*}:Y\rightarrow X_{i}$ in $\mathscr{C}$, there exists a unique anti-morphism $f^{*}:Y\rightarrow X$ such that $f_{i}^{*}=p_{i}\circ f^{*}$, i.e. the following diagram commutes:
$$
\xymatrix{
  X  \ar[r]^{p_{i}} & X_{i}  \\
  Y \ar@{-->}[u]^{f^{*}}\ar[ur]_{f_{i}^{*}} &    }
$$
  \item [(2)]
  Let $(X_{i})_{i\in I}$ be a family of objects of $\mathscr{C}$ indexed by a set $I$ and let $X=\prod_{i\in I}X_{i}$ be the product of $X_{i}$ with anti-projections $p_{i}^{*}:X\rightarrow X_{i}$ ($p_{i}^{*}=p_{i}\circ1^{*}$). Then given any object $Y\in{\rm{Ob}}(\mathscr{C})$ and anti-morphism $f_{i}^{*}:Y\rightarrow X_{i}$ in $\mathscr{C}$, there exists a unique morphism $f:Y\rightarrow X$ such that $f_{i}^{*}=p_{i}^{*}\circ f$, i.e. the following diagram commutes:
$$
\xymatrix{
  X  \ar[r]^{p_{i}^{*}} & X_{i}  \\
  Y \ar@{-->}[u]^{f}\ar[ur]_{f_{i}^{*}} &    }
$$
\end{itemize}
\end{proposition}

Using anti-universal property, we define the notion of \textit{anti-product} associated to a given product, which can be viewed as the anti-analog of the categorical product. By means of anti-universal properties, we prove anti-isomorphism theorem about products and isomorphism theorem about anti-products respectively.

\begin{theorem}Let $\mathscr{C}$ be a factorization category with products.

\begin{itemize}
\item[(1)]
If $(X;p_{i})$ and $(X';p_{i}')$ are both products of the family $(X_{i})_{i\in I}$ in $\mathscr{C}$, then there exists a unique anti-isomorphism $g^{*}:X\xrightarrow{\sim}X'$ such that $p_{i}^{*}=p_{i}'\circ g^{*}$.
\item[(2)]
If $(X;p_{i}^{*})$ and $(X';p_{i}'^{*})$ are both anti-products of the family $(X_{i})_{i\in I}$ in $\mathscr{C}$, then there exists a unique isomorphism $g:X\xrightarrow{\sim}X'$ such that $p_{i}^{*}=p_{i}'^{*}\circ g$.
\end{itemize}
\end{theorem}

In Section \ref{B4}, we define a certain kind of functors between two factorization categories $\mathscr{C}$ and $\mathscr{D}$. Let $F:\mathscr{C}\rightarrow\mathscr{D}$ be a functor. We note that given any $x,y\in\textrm{Ob}(\mathscr{C})$, the set map $\textrm{Hom}_{\mathscr{C}}(x,y)\rightarrow \textrm{Hom}_{\mathscr{D}}(Fx,Fy)$ induces a set map $An_{\mathscr{C}}(x,y)\rightarrow An_{\mathscr{D}}(Fx,Fy)$. In fact, we have the following lemma:
\begin{lemma}
Let $\mathscr{C}$ and $\mathscr{D}$ be factorization categories and let $F:\mathscr{C}\rightarrow\mathscr{D}$ be a functor between the underlying categories. Then for any $x,y\in{\rm{Ob}}(\mathscr{C})$, the set map ${\rm{Hom_{\mathscr{C}}}}(x,y)\rightarrow {\rm{Hom_{\mathscr{D}}}}(Fx,Fy)$ induces a set map $An_{\mathscr{C}}(x,y)\rightarrow An_{\mathscr{D}}(Fx,Fy)$. Moreover, if $\mathscr{C}$ and $\mathscr{D}$ are preadditive factorization categories and $F$ is an additive functor of the underlying preadditive categories, then the group homomorphism ${\rm{Hom_{\mathscr{C}}}}(x,y)\rightarrow {\rm{Hom_{\mathscr{D}}}}(Fx,Fy)$ induces a group homomorphism $An_{\mathscr{C}}(x,y)\rightarrow An_{\mathscr{D}}(Fx,Fy)$.
\end{lemma}

This enables us to define the notion of \textit{factorable functor} between factorization categories. Roughly speaking, a factorable functor is a collection of a functor and all induced set maps between the sets of anti-morphisms that preserves compositions of anti-morphisms and morphisms and compositions of anti-morphisms and anti-morphisms. Then it is a quick consequence that a factorable functor between factorization categories with products preserves anti-products.

Motivated by the example of category $\textbf{Cat}$ of small categories and functors, we form the category $\textbf{Fa}$ of small factorization categories and factorable functors. It is easy to see that one can define functors from $\textbf{Cat}$ to $\textbf{Fa}$ and from $\textbf{Fa}$ to $\textbf{Cat}$. We define the \textit{forgetfully factorial functor} $\textbf{FCA}:\textbf{Fa}\rightarrow \textbf{Cat}$, which forgets the factorial structure of each factorization category. And we define the \textit{factorization functor} $\textbf{CAF}:\textbf{Cat}\rightarrow \textbf{Fa}$, which equips each category with a factorial structure making it a factorization category. Then, we deduce that $\textbf{CAF}$ is both left and right adjoint to $\textbf{FCA}$.

\begin{theorem}
The factorization functor $\textbf{CAF}$ is both left and right adjoint to the forgetfully factorial functor $\textbf{FCA}$. In other words, for $\mathscr{C}\in{\rm{Ob}}(\textbf{Cat})$ and $\mathscr{D}\in{\rm{Ob}}(\textbf{Fa})$, there are functorial bijections
$$
\xi_{\mathscr{C},\mathscr{D}}:{\rm{Hom_{\textbf{Fa}}}}(\textbf{CAF}(\mathscr{C}),\mathscr{D})\xrightarrow{\sim}{\rm{Hom_{\textbf{Cat}}}}(\mathscr{C},\textbf{FCA}(\mathscr{D}))
$$
and
$$
\xi_{\mathscr{D},\mathscr{C}}':{\rm{Hom_{\textbf{Fa}}}}(\mathscr{D},\textbf{CAF}(\mathscr{C}))\xrightarrow{\sim}{\rm{Hom_{\textbf{Cat}}}}(\textbf{FCA}(\mathscr{D}),\mathscr{C}).
$$
\end{theorem}
Moreover, we consider the subcategory $\textbf{Fa-Pa}$ of $\textbf{Fa}$ of small preadditive factorization categories and factorable additive functors. Let $\textbf{Cat-Pa}$ be the subcategory of $\textbf{Cat}$ of $\textbf{Fa}$ of small preadditive categories and additive functors. We can define the \textit{forgetfully preadditive factorial functor} $\textbf{FCA-Pa}:\textbf{Fa-Pa}\rightarrow \textbf{Cat-Pa}$, which can be viewed as the restriction of $\textbf{FCA}$ to $\textbf{FCA-Pa}$. And we can define the \textit{factorization preadditive functor} $\textbf{CAF-Pa}:\textbf{Cat-Pa}\rightarrow \textbf{Fa-Pa}$, which can be viewed as the restriction of $\textbf{CAF}$ to $\textbf{Cat-Pa}$. By the nearly same proof, we can show that $\textbf{CAF-Pa}$ is both left and right adjoint to $\textbf{FCA-Pa}$.
\begin{theorem}
The factorization preadditive functor $\textbf{CAF-Pa}$ is both left and right adjoint to the forgetfully preadditive factorial functor $\textbf{FCA-Pa}$. In other words, for $\mathscr{C}\in{\rm{Ob}}(\textbf{Cat-Pa})$ and $\mathscr{D}\in{\rm{Ob}}(\textbf{Fa-Pa})$, there are functorial bijections
$$
\xi_{\mathscr{C},\mathscr{D}}:{\rm{Hom_{\textbf{Fa-Pa}}}}(\textbf{CAF-Pa}(\mathscr{C}),\mathscr{D})\xrightarrow{\sim}{\rm{Hom_{\textbf{Cat-Pa}}}}(\mathscr{C},\textbf{FCA-Pa}(\mathscr{D}))
$$
and
$$
\xi_{\mathscr{D},\mathscr{C}}':{\rm{Hom_{\textbf{Fa-Pa}}}}(\mathscr{D},\textbf{CAF-Pa}(\mathscr{C}))\xrightarrow{\sim}{\rm{Hom_{\textbf{Cat-Pa}}}}(\textbf{FCA-Pa}(\mathscr{D}),\mathscr{C}).
$$
\end{theorem}

Let $\textbf{Fa-Ab}$ be a subcategory of $\textbf{Fa}$ of small abelian factorization categories and factorable additive functors and let $\textbf{Cat-Ab}$ be a subcategory of $\textbf{Cat}$ of small abelian categories and additive functors. If $\textbf{CAF-Ab}$ is the restriction of $\textbf{CAF-Pa}$ to $\textbf{Cat-Ab}$ and $\textbf{FCA-Ab}$ is the restriction of $\textbf{FCA-Pa}$ to $\textbf{Fa-Ab}$, then by the similar way, we have
\begin{theorem}
The functor $\textbf{CAF-Ab}$ is both left and right adjoint to the functor $\textbf{FCA-Ab}$. In other words, for $\mathscr{C}\in{\rm{Ob}}(\textbf{Cat-Ab})$ and $\mathscr{D}\in{\rm{Ob}}(\textbf{Fa-Ab})$, there are functorial bijections
$$
\xi_{\mathscr{C},\mathscr{D}}:{\rm{Hom_{\textbf{Fa-Ab}}}}(\textbf{CAF-Ab}(\mathscr{C}),\mathscr{D})\xrightarrow{\sim}{\rm{Hom_{\textbf{Cat-Ab}}}}(\mathscr{C},\textbf{FCA-Ab}(\mathscr{D}))
$$
and
$$
\xi_{\mathscr{D},\mathscr{C}}':{\rm{Hom_{\textbf{Fa-Ab}}}}(\mathscr{D},\textbf{CAF-Ab}(\mathscr{C}))\xrightarrow{\sim}{\rm{Hom_{\textbf{Cat-Ab}}}}(\textbf{FCA-Ab}(\mathscr{D}),\mathscr{C}).
$$
\end{theorem}

\section{Anti-homomorphisms of Groups}\label{B1}
In this section, all groups are not necessarily abelian. We first lay down some basic knowledge about anti-homomorphisms of groups. For simplicity, sometimes, we will simply call anti-homomorphism instead of anti-homomorphism of groups.

\begin{definition}[\cite{Chandrasekhara}, Definition 2.1]
An \textit{anti-homomorphism} $\varphi:A\rightarrow B$ of groups, written multiplicatively, is a mapping $\varphi: A\rightarrow B$ such that $\varphi(x\cdot y) = \varphi(y)\cdot\varphi(x)$ for all $x,y\in A$.
\end{definition}

An anti-homomorphism of groups preserves the identity element, inverses, and powers.

\begin{proposition}[\cite{Dr. Dale T. B.}]  \label{BB}
If $\varphi : A \rightarrow B$ is an anti-homomorphism of groups, then $\varphi(1) = 1$, $\varphi(x^{-1}) = \varphi(x)^{-1}$, and $\varphi(x^{n}) = \varphi(x)^{n}$ for all $x\in A$ and $n\in \mathbb{Z}$.
\end{proposition}

Let $A,B,C$ be groups. It is clear that if $\varphi : A \rightarrow B$ is a homomorphism and $\psi : B \rightarrow C$ is an anti-homomorphism, then their composition $\psi \circ \varphi : A \rightarrow C$ is also an anti-homomorphism. And if $\varphi : A \rightarrow B$ is an anti-homomorphism and $\psi : B \rightarrow C$ is a homomorphism, then their composition is an anti-homomorphism as well. However, if $\varphi : A \rightarrow B$ as well as $\psi : B \rightarrow C$ are anti-homomorphisms, then their composition $\psi \circ \varphi : A \rightarrow C$ is a homomorphism. We summarize these to the following proposition.

\begin{proposition}[\cite{Chandrasekhara}, Theorem 2.3, Theorem 2.4]
Let $\varphi:A\rightarrow B$ be a homomorphism (resp. an anti-homomorphism) of groups and $\psi:B\rightarrow C$ be an anti-homomorphism (resp. a homomorphism) of groups. Then $\psi\circ\varphi$ is an anti-homomorphism. If $\varphi:A\rightarrow B$ and $\psi:B\rightarrow C$ are anti-homomorphism of groups, then $\psi\circ\varphi$ is a homomorphism.
\end{proposition}

Since the composition of two anti-homomorphisms is a homomorphism, it is natural to ask if a homomorphism can be expressible as a composite of anti-homomorphisms. The following theorem gives the answer.

\begin{theorem}  \label{AA}
Any homomorphism can be expressed as a composition of anti-homomorphisms.
\end{theorem}

We defer the proof below, due to some undefined notations. Let $A,B,C,D$ be groups and let $\varphi,\varphi_{1},\varphi_{2},\psi_{1},\psi_{2}$ be maps such that $\varphi_{2}\circ\psi_{1}=\psi_{2}\circ\varphi_{1}=\varphi$. Theorem \ref{AA}, in fact, implies that if the diagonal $\varphi$ of the following commutative diagram is a homomorphism, then $\varphi_{1},\varphi_{2},\psi_{1},\psi_{2}$ are anti-homomorphisms.

\begin{figure}[H]
\centering
$$
\xymatrix{
  A \ar[d]_{\psi_{1}} \ar[r]^{\varphi_{1}}\ar@{..>}[dr]^{\varphi} & B \ar[d]^{\psi_{2}} \\
  C \ar[r]_{\varphi_{2}} & D   }
$$
\caption{The Commutative Diagram}
\label{CC}
\end{figure}

For groups and homomorphisms of groups, we can form the category $\textbf{Groups}$ of groups whose morphisms are homomorphisms. However, groups and anti-homomorhisms of groups can not form a category $\mathscr{C}$ whose objects are groups and morphisms are anti-homomorphisms. Since to each triple of objects \textit{A,B,C} of $\mathscr{C}$, the composition $\textrm{Hom}_{\mathscr{C}}(A,B)\times \textrm{Hom}_{\mathscr{C}}(B,C) \rightarrow \textrm{Hom}_{\mathscr{C}}(A,C)$ is not closed (the composition of two anti-homomorphisms is a homomorphism).

For any anti-homomorphism $\varphi : A \rightarrow B$ of groups, if $A$ or $B$ is an abelian group, then the anti-homomorphism becomes a homomorphism, since $\varphi(xy) = \varphi(y)\varphi(x) = \varphi(x)\varphi(y)$, for $x,y\in A$. So the category $\textbf{Ab}$ of abelian groups and homomorphisms can also be viewed as the category of abelian groups and anti-homomorphisms, since there is no differences between homomorphisms and anti-homomorphisms under commutativity.

The following theorem indicates the relation between homomorphisms and anti-homomorphisms.

\begin{theorem}\label{A11}
Let $\varphi : A \rightarrow B$ be an anti-homomorphism of groups. Then $\varphi$ is a homomorphism of groups if $A$ or $B$ is abelian. In particular,
\begin{itemize}
  \item [(1)]
  if $\varphi$ is injective, then $f$ is a homomorphism if and only if $A$ is abelian;
  \item [(2)]
  if $\varphi$ is surjective, then $f$ is a homomorphism if and only if $B$ is abelian;
  \item [(3)]
  moreover, if $\varphi$ is bijective, then the groups $A,B$ are isomorphic abelian.
\end{itemize}
\end{theorem}
\begin{proof}
Without loss of generality, assume that $A$ is abelian, then we have
\begin{align*}
\varphi(yx)&=\varphi(x)\varphi(y), \\
\varphi(xy)&=\varphi(y)\varphi(x), \\
\varphi(x)\varphi(y)&=\varphi(y)\varphi(x) \ (\textrm{since }\varphi(yx)=\varphi(xy)),
\end{align*}
which shows that $\varphi$ is a homomorphism.

Conversely, assume that $\varphi$ is a homomorphism, then for $x,y\in A$
$$
\varphi(xy) = \varphi(x)\varphi(y) = \varphi(y)\varphi(x) = \varphi(yx),
$$
which shows that \textit{B} is abelian if $\varphi$ is surjective. Note that in this case, $A$ is abelian if and only if $\varphi$ is injective. If $\varphi$ is an anti-isomorphism, then $\varphi(xy) = \varphi(yx)$ implies that $xy = yx$, which shows that \textit{A} is also abelian and \textit{A} is isomorphic to \textit{B}.
\end{proof}

In the sequel, we will define a new composition of anti-homomorphisms, so that such composition of two anti-homomorphisms is still an anti-homomorphism without commutativity.

Next, we will define anti-monomorphisms, anti-epimorphisms and anti-isomorphisms. Before that, we first give the following proposition:

\begin{proposition}[\cite{Dr. Dale T. B.}]\label{A15}
If $\varphi$ is a bijective anti-homomorphism of groups, then its inverse bijection $\varphi^{-1}$ is also an anti-homomorphism of groups.
\end{proposition}

\begin{definition}[\cite{Dr. Dale T. B.}]
An \textit{anti-monomorphism} of groups is an injective anti-homomorphism of groups. An \textit{anti-epimorphism} of groups is a surjective anti-homomorphism of groups.

An \textit{anti-isomorphism} of groups is a bijective anti-homomorphism of groups. Two groups \textit{A} and \textit{B} are called \textit{anti-isomorphic} when there exists an anti-isomorphism $A\cong B$. If $f$ is an anti-isomorphism, then its inverse $f^{-1}$ is called its \textit{anti-inverse}.
\end{definition}

Next, we define kernels and images of anti-homomorphisms.

\begin{definition}
Let $\varphi:A\rightarrow B$ be an anti-homomorphism of groups. Then the \textit{kernel} of $\varphi$ is
$$
\textrm{Ker }\varphi:=\{a\in A\mid \varphi(a)=1\}.
$$
The \textit{image} of $\varphi$ is
$$
\textrm{Im }\varphi:=\{\varphi(a)\mid a\in A\}.
$$
\end{definition}

We note some propositions about kernel, cokernel, and image of an anti-homomorphism of groups.
\begin{proposition}
Let $\varphi:A\rightarrow B$ be an anti-homomorphism of groups. Then ${\rm{Ker}}\varphi$ is a normal subgroup of $A$ and ${\rm{Im}}\varphi$ is a subgroup of $B$.
\end{proposition}
\begin{proof}
We just prove that $\textrm{Ker }\varphi$ as a subgroup of $A$ is normal. Since in \cite{Dr. Dale T. B.}, we have $\varphi(x)=\varphi(y)$ if and only if $x-y\in\textrm{Ker }\varphi$ for any $x,y\in A$. Then $x\in\textrm{Ker }\varphi+y$ if and only if $x\in y+\textrm{Ker }\varphi$, which shows that $\textrm{Ker }\varphi$ is normal in $A$. The other parts are in \cite{Dr. Dale T. B.}.
\end{proof}

\begin{proposition}\label{A7}
Let $\varphi:A\rightarrow B$ be an anti-homomorphism of groups. Then $\varphi$ is an anti-monomorphism if and only if ${\rm{Ker }}\varphi=0$. And $\varphi$ is an anti-epimorphism if and only if ${\rm{Coker}} \varphi=0$.
\end{proposition}
\begin{remark}
Note that ${\rm{Coker}} \varphi:=B/{\rm{Im}}\varphi$.
\end{remark}
\begin{proof}
The first part is in \cite{Dr. Dale T. B.}. If $\varphi$ is surjective, then $\textrm{Im }\varphi=B$, which shows that $\textrm{Coker }\varphi=0$. Conversely, if $\textrm{Coker }\varphi=0$, we have $\textrm{Im }\varphi=B$, which shows that $\varphi$ is surjective.
\end{proof}

It is natural to think that if one can define an equivalence relation by means of anti-isomorphisms. However, though being isomorphic is an equivalence relation on groups, we can show that being anti-isomorphic can not form an equivalence relation on groups. In fact, we can define an anti-isomorphism $\varphi : A \rightarrow A$ by $\varphi(x) = x^{-1}$ for $x\in A$ (note that we have $\varphi(xy)=y^{-1}x^{-1}$ for $x,y\in A$), which proves the reflexivity. Then if $\varphi : A \rightarrow B$ is an anti-isomorphism, by Proposition \ref{AA}, $\varphi^{-1}$ is also an anti-isomorphism, and this proves the symmetry. However, since the composition of two anti-homomorphisms is a homomorphism, we see that the transitivity fails.

To offset this flaw, in the following, we will define an equivalence relation on anti-homomorphisms in a way that is not so obvious.

In fact, Theorem \ref{AA} yields such an equivalence relation on anti-homomorphisms. Consider the commutative diagram in Figure \ref{CC}. We can define $(\varphi_{2},\psi_{1})\sim (\varphi_{1},\psi_{2})$, if $\varphi_{2}\circ\psi_{1}$ and $\varphi_{1}\circ\psi_{2}$ are equal to the same homomorphism $\varphi$, i.e. if $(\varphi_{2},\psi_{1})$ and $(\varphi_{1},\psi_{2})$ are factorizations of the same homomorphism $\varphi$. Next, we show that such relation is actually an equivalence relation. Since the reflexivity and symmetry are obvious, we just prove transitivity. If $(\varphi_{3},\psi_{3})$ is another factorization of the homomorphism $\varphi$, then the triple $(\varphi_{2},\psi_{1})$, $(\phi_{2},\psi_{1})$ and $(\phi_{3},\psi_{3})$ are all factorizaitions of the homomorphism $\varphi$, and this proves the transitivity.

We denote by $[(\varphi_{1},\psi_{2})]$ the \textit{equivalence class} of the pair $(\varphi_{1},\psi_{2})$. We use $[\varphi]$ to indicate the \textit{set of all equivalence classes of pairs of anti-homomorphisms whose compositions are the homomorphism} $\varphi$. And the notation $[An(A,B),An(B,C)]$ will stand for the \textit{set of all equivalence classes of all pairs of anti-homomorphisms of the form} $(A\rightarrow B,B\rightarrow C)$.

Since the composition of homomorphisms is associative, the following proposition shows that the composition of homomorphisms and anti-homomorphisms is also associative.

\begin{proposition}\label{A2}
The composition of homomorphisms and anti-homomorphisms is associative. And the composition of anti-homomorphisms is associative.
\end{proposition}
\begin{proof}
We just prove two cases, since the other cases are similar. If $f_{1},f_{2},f_{3}$ are homomorphisms and $f_{1}^{*},f_{2}^{*},f_{3}^{*}$ are anti-homomorphisms. Then when the compositions make sense, for $x$ in the source group, we have
\begin{align*}
[(f_{1}\circ f_{2})\circ f_{3}^{*}](x)=(f_{1}\circ f_{2})(f_{3}^{*}(x))=f_{1}(f_{2}(f_{3}^{*}(x))), \\
[f_{1}\circ (f_{2}\circ f_{3}^{*})](x)=f_{1}((f_{2}\circ f_{3}^{*})(x))=f_{1}(f_{2}(f_{3}^{*}(x))); \\
[(f_{1}\circ f_{2}^{*})\circ f_{3}^{*}](x)=(f_{1}\circ f_{2}^{*})(f_{3}^{*}(x))=f_{1}(f_{2}^{*}(f_{3}^{*}(x))), \\
[f_{1}\circ (f_{2}^{*}\circ f_{3}^{*})](x)=f_{1}((f_{2}^{*}\circ f_{3}^{*})(x))=f_{1}(f_{2}^{*}(f_{3}^{*}(x))).
\end{align*}
These show that $(f_{1}\circ f_{2})\circ f_{3}^{*}=f_{1}\circ (f_{2}\circ f_{3}^{*})$ and $(f_{1}\circ f_{2}^{*})\circ f_{3}^{*}=f_{1}\circ (f_{2}^{*}\circ f_{3}^{*})$.
\end{proof}

In fact, we can view any anti-homomorphism as a composition of a homomorphism and a reverse mapping. The \textit{reverse mapping} $f:G\rightarrow G$ of a group $G$ is a mapping $f:G\rightarrow G$ such that $f(x)=x$ and $f(xy)=yx$ for all $x,y\in G$. Assume that $\varphi : S \rightarrow H$ is a homomorphism of groups and $\psi : H \rightarrow H$ is a reverse mapping. Then $\psi \circ \varphi: S \rightarrow H$, $\psi \circ \varphi(xy) = \psi(\varphi(x)\varphi(y)) = \varphi(y)\varphi(x)$ for $x,y\in S$, shows that the composition becomes an anti-homomorphism. And this indeed proves Theorem \ref{AA}, but we can simplify the notation.

We summarize the results that are of special interest to us in the following proposition:

\begin{definition-proposition}\ \label{A3}

\begin{itemize}
  \item [(1)]
  To each pair of objects \textit{A,B} $\in {\rm{Ob}}(\textbf{Groups})$, we denote the {\rm{set of homomorphisms}} by ${\rm{Hom}}(A,B)$ and the {\rm{set of anti-homomorphisms}} by $An(A,B)$. And we denote the {\rm{set of isomorphisms}} by ${\rm{Hom.Is}}(A,B)$ and the {\rm{set of anti-isomorphisms}} by $An{\rm{.Is}}(A,B)$.
  \item [(2)]
  To each triple of objects \textit{A,B,C} $\in {\rm{Ob}}(\textbf{Groups})$, we have \textit{laws of compositions}
  \begin{align*}
  &An(A,B) \times An(B,C) \rightarrow {\rm{Hom}}(A,C)\\
  &An(A,B) \times {\rm{Hom}}(B,C) \rightarrow An(A,C)\\
  &{\rm{Hom}}(A,B) \times An(B,C) \rightarrow An(A,C)\\
  &{\rm{Hom}}(A,B) \times {\rm{Hom}}(B,C) \rightarrow {\rm{Hom}}(A,C)
  \end{align*}
  and \textit{a law of factorization}
  $$
  {\rm{Hom}}(A,C) \rightarrow [An(A,B),An(B,C)], \ \ f\mapsto [f].
  $$
  \item [(3)]
  For every $A,B\in {\rm{Ob}}(\textbf{Groups})$, there exists a reverse mapping $1^{*}_{A}\in An(A,A)$ such that for any morphism $f\in {\rm{Hom}}(A,B)$, we have
  $f\circ 1^{*}_{A}\in An(A,B)$. And for any $g\in {\rm{Hom}}(B,A)$, we have $1^{*}_{A}\circ f\in An(B,A)$.
  \item [(4)]
  The compositions defined above are \textit{associative}.
\end{itemize}
\end{definition-proposition}
\begin{remark}
Note that for every $A,B\in \textrm{Ob}(\textbf{Groups})$ and for any $f^{*}\in An(A,B)$, we have $f^{*}\circ 1^{*}_{A}\in \textrm{Hom}(A,B)$. And for any $g^{*}\in An(B,A)$, we have $1^{*}_{A}\circ g^{*}\in \textrm{Hom}(B,A)$. And for the reverse mapping $1^{*}_{A}$, we will suppress the subscript and simply write $1^{*}$ if there is no confusion.
\end{remark}

\begin{proposition}\label{A1}
For the category $\textbf{Groups}$ of groups, we have a one-to-one correspondence of sets ${\rm{Hom}}(A,B)\xrightarrow{\sim} An(A,B)$ for any $A,B\in {\rm{Ob}}(\textbf{Groups})$.
\end{proposition}
\begin{proof}
In fact, we can define a function $\textrm{Hom}(A,B)\rightarrow An(A,B),f\mapsto f^{*}=f\circ1^{*}$, whose inverse is the function defined by $An(A,B)\rightarrow\textrm{Hom}(A,B),f^{*}\mapsto f=f^{*}\circ1^{*}$. So we get the one-to-one correspondence $\textrm{Hom}(A,B)\xrightarrow{\sim} An(A,B)$ as desired.
\end{proof}
\begin{remark}
For any homomorphism $f$, the anti-homomorphism $f^{*}=f\circ1^{*}=1^{*}\circ f$ is called the \textit{corresponding anti-homomorphism of} $f$. Since every anti-homomorphism $\varphi$ can be written as $f\circ1^{*}=1^{*}\circ f$ for some homomorphism $f$, we can write $f^{*}$ for any anti-homomorphism. Similarly, for any anti-homomorphism $f^{*}$, the homomorphism $f=f^{*}\circ1^{*}=1^{*}\circ f^{*}$ is called the \textit{corresponding homomorphism of} $f^{*}$. Clearly, we have $f=f^{**}$.
\end{remark}

Then \cite{Pursell}, Theorem 7 in the case of groups becomes a corollary of Proposition \ref{A1}.
\begin{corollary}
If $A,B\in {\rm{Ob}}(\textbf{Groups})$, then $A,B$ are isomorphic if and only if $A,B$ are anti-isomorphic.
\end{corollary}

With the notation defined in Proposition \ref{A3}, we are able to give a brief proof of Theorem \ref{AA}.
\begin{proof}[Proof of Theorem \ref{AA}]
Let $\phi:A\rightarrow B$ be a homomorphism of groups. Then there is a reverse mapping $1^{*}:A\rightarrow A$ of $A$ and an anti-homomorphism $\varphi:A\rightarrow B$ such that $(\varphi\circ1^{*})(xy)=\varphi(1^{*}(xy))=\varphi(yx)=\varphi(x)\varphi(y)$ for $x,y\in A$. Thus we have $\phi=\varphi\circ1^{*}$ as desired.
\end{proof}

\begin{corollary}
Every homomorphism $\phi:A\rightarrow B$ of groups can be factored into a sequence
$$
A\xrightarrow{1^{*}} A\xrightarrow{\varphi} B,
$$
where $1^{*}$ is a reverse mapping of $A$ and $\varphi$ is an anti-homomorphism.
\end{corollary}

With the notation introduced in Proposition \ref{A3}, we deduce from Theorem \ref{A11} the following proposition, which is \cite{Pursell}, Theorem 8 in the case of groups:
\begin{proposition}\label{A13}
If a group $A$ is not abelian, then ${\rm{Hom.Is}}(A,B)\cap An{\rm{.Is}}(A,B)=\varnothing$ for any $B\in {\rm{Ob}}(\textbf{Groups})$.
\end{proposition}
\begin{proof}
If $f\in{\rm{Hom.Is}}(A,B)\cap An{\rm{.Is}}(A,B)$, then by Theorem \ref{A11}, $B$ is abelian. However, $A$ is not abelian but $A\cong B$, which is a contradiction! So we have ${\rm{Hom.Is}}(A,B)\cap An{\rm{.Is}}(A,B)=\varnothing$.
\end{proof}

Next, we prove an analogue of the factorization theorem (cf. \cite{Grillet}, Chapter I, Theorem 5.1), which is called \textit{the anti-factorization theorem}.

\begin{theorem}[Anti-factorization Theorem]\label{A4}
Let $N$ be a normal subgroup of $G$ and $\varphi:G\rightarrow H$ an anti-homomorphism whose kernel contains $N$, then there exists a unique anti-homomorphism $\psi:G/N\rightarrow H$ such that the following diagram commutes ($\pi$ is the natural projection):
$$
\xymatrix{
  G \ar[dr]_{\varphi} \ar[r]^{\pi}
                & G/N \ar@{-->}[d]^{\exists!\psi}  \\
                & H             }
$$
\end{theorem}
\begin{proof}
First, we define $\psi(xN):=\varphi(x)$ for all $x\in G$. Then we want to show that $\psi$ is well-defined. Since $\textrm{Ker }\varphi$ contains $N$, if $x^{-1}y\in N$, we have $\varphi(y)\varphi(x^{-1})=\varphi(x^{-1}y)=1$, so that $xN=yN$ implies $\varphi(x)=\varphi(y)$. Thus we have $\varphi=\psi\circ\pi$.

Then since $\varphi$ is an anti-homomorphism, we can easily show that $\psi$ is also an anti-homomorphism. Next, we want to show that $\psi$ is unique. Let $\ell:G/N\rightarrow H$ be another anti-homomorphism such that $\varphi=\ell\circ\pi$. Then $\ell(xN)=\varphi(x)=\psi(xN)$ for all $xN\in G/N$, which shows that $\ell=\psi$.
\end{proof}

The following theorem can also be viewed as an analogue of the homomorphism theorem, but since its proof is not analogical to that of the homomorphism theorem, we would not call it the anti-homomorphism theorem.
\begin{theorem}
Let $\varphi:A\rightarrow B$ be a homomorphism of groups, then there exists a canonical anti-isomorphism
$$
A/{\rm{Ker }}\varphi\cong {\rm{Im }}\varphi.
$$
\end{theorem}
\begin{proof}
The Homomorphism Theorem yields a canonical isomorphism $\varphi:A/\textrm{Ker }\varphi\cong\textrm{Im }\varphi$. Then composing with the reverse mapping $1^{*}$, we get the anti-isomorphism $\varphi\circ1^{*}:A/\textrm{Ker }\varphi\cong \textrm{Im }\varphi$ as desired.
\end{proof}

In the following, we will prove the anti-analogues of the homomorphism theorem, the second isomorphism theorem, and the third isomorphism theorem (see \cite{Grillet}, Chapter I, Theorem 5.2, Theorem 5.8, and Theorem 5.9). These results may have been proved in \cite{Gopalakrishnamoorthy}, but the literature is lost. So we reprove these important results.

\begin{theorem}[The Anti-homomorphism Theorem]\label{A5}
If $\varphi:A\rightarrow B$ is an anti-homomorphism, then there is a canonical anti-isomorphism
$$
A/{\rm{Ker}}\varphi\cong {\rm{Im}}\varphi,
$$
that is, there is a unique anti-isomorphism $\xi:A/{\rm{Ker}}\varphi\xrightarrow{\sim} {\rm{Im}}\varphi$ such that the following diagram commutes:
$$
\xymatrix{
  A \ar[d]_{\pi} \ar[r]^{\varphi} & B  \\
  A/{\rm{Ker }}\varphi \ar@{-->}[r]^{\exists!\xi} & {\rm{Im }}\varphi \ar[u]^{\imath}  }
$$
($\imath$ is the inclusion homomorphism and $\pi$ is the natural projection.)
\end{theorem}
\begin{proof}
Let $\psi:A\rightarrow \textrm{Im }\varphi$ be an anti-homomorphism such that $\varphi=\imath\circ\psi$. Then $\textrm{Ker }\psi=\textrm{Ker }\varphi$, and Theorem \ref{A4} shows that there exists a unique anti-homomorphism $\xi:A/\textrm{Ker }\psi\rightarrow \textrm{Im }\varphi$ such that $\psi=\xi\circ\pi$. Hence we have $\varphi=\imath\circ\xi\circ\pi$. Moreover, the surjectivity of $\psi$ implies that $\xi$ is surjective. And $\xi$ is injective since $\xi(x\textrm{Ker }\psi)=1\Rightarrow \varphi(x)=1\Rightarrow x\in\textrm{Ker }\psi$. Hence, $\xi$ is an anti-isomorphism.

Next, we prove the uniqueness. If $\ell$ is another anti-homomorphism such that $\varphi=\imath\circ\ell\circ\pi$. Then we have $\ell(x\textrm{Ker }\varphi)=\varphi(x)=\xi(x\textrm{Ker }\varphi)$ for all $x\textrm{Ker }\varphi\in A/\textrm{Ker }\varphi$, which shows that $\ell=\xi$.
\end{proof}

\begin{corollary}
Let $\varphi:A\rightarrow B$ be an anti-homomorphism. If $\varphi$ is injective, then $A\cong {\rm{Im }}\varphi$. If $\varphi$ is surjective, then $B\cong A/{\rm{Ker }}\varphi$.
\end{corollary}
\begin{proof}
If $\varphi$ is injective, then $\textrm{Ker }\varphi=1$, we have $A\cong A/\textrm{Ker }\varphi\cong \textrm{Im }\varphi$. If $\varphi$ is surjective, then $B=\textrm{Im }\varphi\cong A/\textrm{Ker }\varphi$.
\end{proof}

\begin{theorem}[The Second Anti-isomorphism Theorem]
Let $A$ be a group and $B,C$ be normal subgroups of $A$. If $C\subset B$, then $C$ is a normal subgroup of $B$ and $B/C$ is a normal subgroup of $A/C$. Then there exists a unique anti-isomorphism
$$
A/B\cong(A/C)/(B/C),
$$
that is, there exists a unique anti-isomorphism $\xi:A/B\xrightarrow{\sim}(A/C)/(B/C)$ such that the following diagram commutes:
$$
\xymatrix{
  A \ar[d]_{\rho^{*}} \ar[rr]^{\pi} &    & A/C \ar[d]^{\tau}\ar[lld]_{\sigma} \\
  A/B \ar@{-->}[rr]^{\exists!\xi} &   & (A/C)/(B/C)   }
$$
($\pi,\tau$ are natural projections and $\rho^{*}$ is an anti-projection, i.e. for the natural projection $\rho$, $\rho^{*}=\rho\circ1^{*}$).
\end{theorem}
\begin{proof}
By Theorem \ref{A4}, there exists some unique anti-homomorphism $\sigma:A/C\rightarrow A/B, aC\mapsto aB$ such that $\rho^{*}=\sigma\circ\pi$. $\sigma$ is surjective since $\pi$ and $\rho^{*}$ are surjective. Next, we want to show that $\textrm{Ker }\sigma=B/C$. First, if $bC\in B/C$, where $b\in B$, then we have $\sigma(bC)=bB=1$ in $A/B$. Conversely, if $\sigma(aC)=aB=1$, we have $a\in B$. So $\textrm{Ker }\sigma=B/C$ and $\textrm{Ker }\sigma$ is a normal subgroup of $A/C$. By Theorem \ref{A5}, $A/B=\textrm{Im }\sigma\cong (A/C)/\textrm{Ker }\sigma=(A/C)/(B/C)$. In fact, Theorem \ref{A5} yields an anti-isomorphism $\xi:A/B\cong(A/C)/(B/C)$ such that $\tau=\xi\circ\sigma$, and then $\xi\circ\rho^{*}=\tau\circ\pi$.

If $\ell$ is another anti-isomorphism such that $\xi\circ\rho^{*}=\tau\circ\pi$, then the surjectivity of $\rho^{*}$ implies that $\ell=\xi$.
\end{proof}

\begin{theorem}[The Third Anti-isomorphism Theorem]
Let $A$ be a subgroup of a group $G$, and let $N$ be a normal subgroup of $G$. Then $AN$ is a subgroup of $G$, $N$ is a normal subgroup of $AN$, $A\cap N$ is a normal subgroup of $A$, and there exists a unique anti-isomorphism
$$
AN/N\cong A/(A\cap N),
$$
that is, there is a unique anti-isomorphism $\xi:AN/N\xrightarrow{\sim} A/(A\cap N)$ such that the following diagram commutes:
$$
\xymatrix{
  A \ar[d]_{\imath}\ar[rrd]^{\varphi} \ar[rr]^{\rho} &   & A/(A\cap N) \ar@{-->}[d]^{\exists!\xi} \\
  AN \ar[rr]^{\pi^{*}} &  & AN/N   }
$$
($\imath$ is the inclusion homomorphism, $\rho$ is the natural projection, and $\pi^{*}$ is the anti-projection).
\end{theorem}
\begin{proof}
First, we show that $AN$ is a subgroup of $G$. Since $N$ is a normal subgroup of $G$, we have $AN=NA$. Thus, $an\in AN$ for $a\in A,n\in N$ implies that $(an)^{-1}=n^{-1}a^{-1}\in NA=AN$. And clearly we have $1\in AN$ and $ANAN=AANN=AN$. So $AN$ is a subgroup of $G$. Then, we see that $N$ is a normal subgroup of $AN$.

Let $\varphi=\pi^{*}\circ\imath$. Then $\varphi(a)=aN\in AN/N$ for every $a\in A$, and $\varphi$ is surjective. Since $\varphi(a)=1$ if and only if $a\in N$, i.e. if and only if $a\in A\cap N$, $\textrm{Ker }\varphi=A\cap N$. By Theorem \ref{A5}, $AN/N=\textrm{Im }\varphi\cong A/\textrm{Ker }\varphi=A/(A\cap N)$. In fact, Theorem \ref{A5} yields a unique anti-isomorphism $\xi:A/(A\cap N)\cong AN/N$ such that $\xi\circ\rho=\pi^{*}\circ\imath$, since $\varphi=\pi^{*}\circ\imath$.
\end{proof}

Next, we would like to define a kind of composition called $*$-composition, which is different to the composition defined in \cite{Gopalakrishnamoorthy}, such that the $*$-composition of two anti-homomorphisms is still an anti-homomorphism.

\begin{definition}
Let $\varphi:A\rightarrow B,\psi:B\rightarrow C$ be two anti-homomorphisms of groups. We define $\psi*\varphi:=(\psi\circ\varphi)\circ1^{*}:A\rightarrow C$ to be the \textit{$*$-composition} of $\psi$ and $\varphi$.
\end{definition}

\begin{proposition}
The $*$-composition is well-defined.
\end{proposition}
\begin{proof}
If $\varphi_{1},\varphi_{2}:A\rightarrow B,\psi_{1},\psi_{2}:B\rightarrow C$ are anti-homomorphisms of groups with $\varphi_{1}=\varphi_{2}$ and $\psi_{1}=\psi_{2}$, then $\varphi_{1}*\psi_{1}=(\psi_{1}\circ\varphi_{1})\circ1^{*}=(\psi_{2}\circ\varphi_{2})\circ1^{*}=\varphi_{2}*\psi_{2}$.
\end{proof}

\begin{proposition}
Let $\varphi:A\rightarrow B,\psi:B\rightarrow C$ be two anti-homomorphisms. Then $\psi*\varphi$ is an anti-homomorphism.
\end{proposition}
\begin{proof}
By definition, $\psi\circ\varphi$ is a homomorphism. Then by Proposition \ref{A3}, composing with $1^{*}$, we get an anti-homomorphism $\psi*\varphi=(\psi\circ\varphi)\circ1^{*}:A\rightarrow C$.
\end{proof}

\begin{proposition}\label{A32}
The $*$-composition of anti-homomorphisms is associative.
\end{proposition}
\begin{proof}
Let $\varphi_{1},\varphi_{2},\varphi_{3}$ be three anti-homomorphisms. Then we have
\begin{align*}
(\varphi_{1}*\varphi_{2})*\varphi_{3}&=[(\varphi_{1}\circ\varphi_{2})\circ1^{*}\circ\varphi_{3}]\circ1^{*}, \\
\varphi_{1}*(\varphi_{2}*\varphi_{3})&=[\varphi_{1}\circ(\varphi_{2}\circ\varphi_{3})\circ1^{*}]\circ1^{*}.
\end{align*}
By Proposition \ref{A2} and Proposition \ref{A3}, we get $(\varphi_{1}*\varphi_{2})*\varphi_{3}=\varphi_{1}*(\varphi_{2}*\varphi_{3})$ as desired.
\end{proof}

\begin{proposition}
If $f$ is an anti-isomorphism and $f^{-1}$ its anti-inverse, then $f*f^{-1}=f^{-1}*f=1^{*}$.
\end{proposition}
\begin{proof}
By definition, we have
$$
f*f^{-1}=(f\circ f^{-1})\circ1^{*}=1\circ1^{*}=1^{*},
$$
and
$$
f^{-1}*f=(f^{-1}\circ f)\circ1^{*}=1\circ1^{*}=1^{*}.
$$
\end{proof}

\begin{proposition}\label{A33}
If $f$ is an anti-homomorphism, then $f*1^{*}=1^{*}*f=f$.
\end{proposition}
\begin{proof}
By definition,
$$
f*1^{*}=(f\circ1^{*})\circ1^{*}=(1^{*}\circ f)\circ1^{*}=1^{*}*f=f.
$$
\end{proof}

By virtue of Proposition \ref{A32} and \ref{A33}, we can finally form the category $\textbf{Groups-An}$ of groups and anti-homomorphisms, i.e. the category whose objects are groups, morphisms are anti-homomorphisms, and the composition of anti-homomorphisms is $*$-composition. Note that in $\textbf{Groups-An}$, reverse mappings will serve as the identity morphisms.

Next, we want to clarify the relation between $\textbf{Groups}$ and $\textbf{Groups-An}$.

One can define a functor
$$
F_{An}:\textbf{Groups}\rightarrow \textbf{Groups-An}
$$
which assigns to every group the same group, and to every homomorphism $f$, it associates the corresponding anti-homomorphism $f^{*}$.

And one can define a functor
$$
G_{An}:\textbf{Groups-An}\rightarrow \textbf{Groups}
$$
which assigns to every group the same group, and to every anti-homomorphism $f^{*}$, it associates the corresponding homomorphism $f$.

\begin{theorem}
The category \textbf{Groups} of groups and homomorphisms and the category \textbf{Groups-An} of groups and anti-homomorphisms are equivalent, i.e. $\textbf{Groups}\cong \textbf{Groups-An}$.
\end{theorem}
\begin{proof}
Clearly every objects $A$ of \textbf{Groups-An} is isomorphic to $F_{An}(A)$. And by Proposition \ref{A1}, we see that $F_{An}$ is fully faithful. Hence $F_{An}$ is an equivalence of categories.
\end{proof}

Let $A$ be any group. We can define a group structure on $\textrm{Hom.Is}(A,A)$ using usual composition. Also we can define a group structure $An\textrm{.Is}(A,A)$ using $*$-composition for $A\in \textrm{Ob}(\textbf{Groups})$. Then we obtain two isomorphic groups $(\textrm{Hom.Is}(A,A),\circ)$ and $(An\textrm{.Is}(A,A),*)$.

\begin{proposition}\label{A14}
The groups $({\rm{Hom.Is}}(A,A),\circ)$ and $(An{\rm{.Is}}(A,A),*)$ are isomorphic.
\end{proposition}
\begin{proof}
Let $f,g\in\textrm{Hom.Is}(A,A)$. Then by Proposition \ref{A1}, $f\circ g\mapsto (f\circ g)^{*}$. But
$$
f^{*}*g^{*}=(f\circ1^{*}\circ g\circ1^{*})\circ1^{*}=(f\circ g)\circ1^{*}=(f\circ g)^{*},
$$
so we get the desired isomorphism.
\end{proof}

Note that in Proposition \ref{A14}, the group $({\rm{Hom.Is}}(A,A),\circ)$ is indeed the automorphism group $(\textrm{Aut}(G),\circ)$, whose elements are automorphisms of $G$, and the multiplication is the usual composition. And the group $(An{\rm{.Is}}(A,A),*)$ can be regraded as the automorphism group whose elements are anti-automorphisms and the multiplication is the $*$-composition. We will call $(An{\rm{.Is}}(A,A),*)$ the \textit{anti-automorphism group}.

If $(An+\textrm{Hom}).\textrm{Is}(A,A):=An{\rm{.Is}}(A,A)\cup{\rm{Hom.Is}}(A,A)$ is equipped with the usual composition, then we claim that $(An+\textrm{Hom}).\textrm{Is}(A,A)$ is group with a normal subgroup ${\rm{Hom.Is}}(A,A)$. We indeed prove \cite{Pursell}, Theorem 9 in the case of group.
\begin{proposition}
$(An+{\rm{Hom}}).{\rm{Is}}(A,A)$ is a group under the usual composition with a normal subgroup ${\rm{Hom.Is}}(A,A)$.
\end{proposition}
\begin{proof}
By Proposition \ref{A2} and Proposition \ref{A15}, it is easy to see that $An{\rm{.Is}}(A,A)\cup{\rm{Hom.Is}}(A,A)$ is a group under usual composition. And let $f_{1}^{*}\in An{\rm{.Is}}(A,A),f\in{\rm{Hom.Is}}(A,A)$, since $f_{1}^{*}\circ f=f_{1}\circ1^{*}\circ f=f_{1}\circ f^{*}$, we see that ${\rm{Hom.Is}}(A,A)$ is a normal subgroup of $An{\rm{.Is}}(A,A)\cup{\rm{Hom.Is}}(A,A)$.
\end{proof}

There is clearly an obvious way to define addition and multiplication on $\textrm{Hom}(A,B)$ and $An(A,B)$ for $A,B\in \textrm{Ob}(\textbf{Groups})$, for which we feel that it is convenient to discuss it in the next section.

\section{Anti-homomorphisms of Rings}\label{B2}
In this section, all rings are not necessarily commutative. By virtue of \cite{Poonen}, all rings should have the identity. We first lay down some basic knowledge about anti-homomorphisms of rings. For simplicity, sometimes, we will simply call anti-homomorphism rather than anti-homomorphism of rings. Some results about anti-homomorphisms of groups can be directly moved to anti-homomorphisms of rings.

\begin{definition}[\cite{Dr. Dale T. B.}]
An \textit{anti-homomorphism} $\varphi:A\rightarrow B$ of rings is a map $\varphi:A\rightarrow B$ such that $\varphi(x+y)=\varphi(x)+\varphi(y)$, $\varphi(xy)=\varphi(y)\varphi(x)$, and $\varphi(1)=1$ for all $x,y\in A$.
\end{definition}

An anti-homomorphism of rings preserves the zero element, integer multiples, and powers.
\begin{proposition}
If $\varphi : A \rightarrow B$ is an anti-homomorphism of rings, then $\varphi(0)=0$, $\varphi(mx) = m\varphi(x)$, and $\varphi(x^{n}) = \varphi(x)^{n}$ for all $x\in A$, $m\in\Z$, and $n\in\mathbb{N}$.
\end{proposition}
\begin{proof}
Since $\varphi : A \rightarrow B$ is also an anti-homomorphism of abelian groups, we have $\varphi(mx) = m\varphi(x)$ for $m\in\Z$. The other parts is in \cite{Dr. Dale T. B.}.
\end{proof}

Before providing the definitions of anti-monomorphisms, anti-epimorphisms, and anti-isomorphisms, we first give the following proposition.

\begin{proposition}[\cite{Dr. Dale T. B.}]
If $\varphi$ is a bijective anti-homomorphism of rings, then its inverse bijection $\varphi^{-1}$ is also an anti-homomorphism of rings.
\end{proposition}

\begin{definition}[\cite{Dr. Dale T. B.}]
An \textit{anti-monomorphism} of rings is an injective anti-homomorphism of rings. An \textit{anti-epimorphism} of rings is a surjective anti-homomorphism of rings.

An \textit{anti-isomorphism} of rings is a bijective anti-homomorphism of rings. Two rings \textit{A} and \textit{B} are called \textit{anti-isomorphic} when there exists an anti-isomorphism $A\cong  B$. If $f$ is an anti-isomorphism of rings, then its inverse $f^{-1}$ is called its \textit{anti-inverse}.
\end{definition}

We can define the kernel and image of an anti-homomorphism.
\begin{definition}[\cite{Dr. Dale T. B.}]
Let $\varphi:A\rightarrow B$ be an anti-homomorphism of rings. Then the \textit{kernel} of $\varphi$ is
$$
\textrm{Ker }\varphi:=\{a\in A\mid \varphi(a)=0\}.
$$
The \textit{image} of $\varphi$ is
$$
\textrm{Im }\varphi:=\{\varphi(a)\mid a\in A\}.
$$
\end{definition}

The following proposition shows that an anti-homomorphism of rings preserves subrings and an anti-epimorphism reverses ideals.
\begin{proposition}
Let $\varphi:A\rightarrow B$ be an anti-homomorphism of rings. If $A_{0}$ is a subring of $A$, then
$$
\varphi(A_{0})=\{\varphi(x)\mid x\in A_{0}\}
$$
is a subring of $B$. If $B_{0}$ is a subring of $B$, then
$$
\varphi^{-1}(B_{0})=\{x\in A\mid \varphi(x)\in B_{0}\}
$$
is a subring of $A$. Moreover, if $\mathfrak{a}$ is a left-ideal of $A$ and $\varphi$ is an anti-epimorphism, then $\varphi(\mathfrak{a})$ is a right-ideal of $B$. And if $\mathfrak{b}$ is a left-ideal of $B$ and $\varphi$ is an anti-epimorphism, then $\varphi^{-1}(\mathfrak{b})$ is a right-ideal of $A$.
\end{proposition}
\begin{proof}
Let $a\in\mathfrak{a}$ such that $\varphi(a)\in\varphi(\mathfrak{a})$. Since $\varphi$ is an anti-epimorphism, for every $b\in B$, there exists $a_{0}\in A$ such that $\varphi(a_{0})=b$. So we have $\varphi(a)b=\varphi(a)\varphi(a_{0})=\varphi(a_{0}a)\in\varphi(\mathfrak{a})$, which shows that $\varphi(\mathfrak{a})$ is a right-ideal of $B$.

Let $b\in \mathfrak{b}$ such that $\varphi^{-1}(b)\subset\varphi^{-1}(\mathfrak{b})$. Since $\mathfrak{b}$ is a left-ideal of $B$, for every $b'\in B$, we have $b'b\in\mathfrak{b}$. So there exist $a',a\in A$ with $\varphi(a')=b',\varphi(a)=b$ such that $\varphi(aa')=\varphi(a')\varphi(a)=b'b\in\mathfrak{b}$, which implies that $aa'\in\varphi^{-1}(\mathfrak{b})$ and that $\varphi^{-1}(\mathfrak{b})$ is a right-ideal of $A$.
\end{proof}

It is easy to see that the composition of an anti-homomorphism and a homomorphism is an anti-homomorphism. And the composition of two anti-homomorphisms is a homomorphism. We summarize these to the following proposition.

\begin{proposition}[\cite{Dr. Dale T. B.}]
Let $\varphi:A\rightarrow B$ be a homomorphism (resp. an anti-homomorphism) of rings and $\psi:B\rightarrow C$ be an anti-homomorphism (resp. a homomorphism) of rings. Then $\psi\circ\varphi$ is an anti-homomorphism. If $\varphi:A\rightarrow B$ and $\psi:B\rightarrow C$ are anti-homomorphism of rings, then $\psi\circ\varphi$ is a homomorphism.
\end{proposition}

Since the composition of two anti-homomorphisms is a homomorphism, it is natural to ask if a homomorphism can be expressible as a composite of anti-homomorphisms. The following theorem answers the question.

\begin{theorem}\label{A8}
Any homomorphism can be expressed as a composition of anti-homomorphisms.
\end{theorem}

We also defer the proof below. In the following, we provide some propositions that hold in the case of groups.

\begin{proposition}
Let $\varphi:A\rightarrow B$ be an anti-homomorphism of rings. Then $\varphi$ is an anti-mono\-morphism if and only if ${\rm{Ker }}\varphi=0$. And $\varphi$ is an anti-epimorphism if and only if ${\rm{Coker}} \varphi=0$.
\end{proposition}
\begin{proof}
The first part is in \cite{Dr. Dale T. B.}. The rest of the proof is the same as Proposition \ref{A7}.
\end{proof}

\begin{proposition}
Let $\varphi:A\rightarrow B$ be an anti-homomorphism of rings. Then ${\rm{Ker}}\varphi$ is an ideal of $A$ and ${\rm{Im}}\varphi$ is a subring of $B$.
\end{proposition}
\begin{proof}
By definition, if $a,b\in\textrm{Ker }\varphi$, then $a-b\in\textrm{Ker }\varphi$, which shows that $\textrm{Ker }\varphi$ is a subgroup of $A$. And if $r\in A,a\in\textrm{Ker }\varphi$, we have $\varphi(ra)=\varphi(a)\varphi(r)=0=\varphi(r)\varphi(a)=\varphi(ar)$, which shows that $ra,ar\in\textrm{Ker }\varphi$. Hence $\textrm{Ker }\varphi$ is an ideal of $A$. The rest is in \cite{Dr. Dale T. B.}.
\end{proof}

\begin{proposition}
The composition of homomorphisms and anti-homomorphisms is associative. And the composition of anti-homomorphisms is associative.
\end{proposition}
\begin{proof}
The proof is the same as Proposition \ref{A2}.
\end{proof}

In fact, Theorem \ref{A8} yields an equivalence relation on anti-homomorphisms of rings. Let $(\varphi_{2},\psi_{1})$ and $(\varphi_{1},\psi_{2})$ be two pairs of anti-homomorphisms. We can define $(\varphi_{2},\psi_{1})\sim (\varphi_{1},\psi_{2})$, if $\varphi_{2}\circ\psi_{1}$ and $\varphi_{1}\circ\psi_{2}$ are equal to the same homomorphism $\varphi$ when the compositions make sense. It can be easily checked that we indeed define an equivalence relation. We denote by $[(\varphi_{1},\psi_{2})]$ the \textit{equivalence class} of the pair $(\varphi_{1},\psi_{2})$. We use $[\varphi]$ to indicate the \textit{set of all equivalence classes of pairs of anti-homomorphisms whose compositions are the homomorphism} $\varphi$. And the notation $[An(A,B),An(B,C)]$ will stand for the \textit{set of all equivalence classes of all pairs of anti-homomorphisms of the form} $(A\rightarrow B,B\rightarrow C)$.

The \textit{reverse mapping} $f$ of a ring $R$ is a mapping $f:R\rightarrow R$ such that $f(x)=x$, $f(x+y)=x+y$, and $f(xy)=yx$ for all $x,y\in R$.

We summary the results that are of particular interest to us in the following:

\begin{definition-proposition} \label{A12}
Consider the category $\textbf{Rings}$ of rings and homomorphisms.
\begin{itemize}
  \item [(1)]
  To each pair of objects \textit{A,B} $\in {\rm{Ob}}(\textbf{Rings})$, we denote the {\rm{set of homomorphisms}} by ${\rm{Hom}}(A,B)$ and the {\rm{set of anti-homomorphisms}} by $An(A,B)$. And we denote the {\rm{set of isomorphisms}} by ${\rm{Hom.Is}}(A,B)$ and the {\rm{set of anti-isomorphisms}} by $An{\rm{.Is}}(A,B)$.
  \item [(2)]
  To each triple of objects \textit{A,B,C} $\in {\rm{Ob}}(\textbf{Rings})$, we have \textit{laws of compositions}
  \begin{align*}
  &An(A,B) \times An(B,C) \rightarrow {\rm{Hom}}(A,C)\\
  &An(A,B) \times {\rm{Hom}}(B,C) \rightarrow An(A,C)\\
  &{\rm{Hom}}(A,B) \times An(B,C) \rightarrow An(A,C)\\
  &{\rm{Hom}}(A,B) \times {\rm{Hom}}(B,C) \rightarrow {\rm{Hom}}(A,C)
  \end{align*}
  and \textit{a law of factorization}
  $$
  {\rm{Hom}}(A,C) \rightarrow [An(A,B),An(B,C)], \ \ f\mapsto [f].
  $$
  \item [(3)]
  For every $A,B\in {\rm{Ob}}(\textbf{Rings})$, there exists a reverse mapping $1^{*}_{A}\in An(A,A)$ such that for any morphism $f\in {\rm{Hom}}(A,B)$, we have
  $f\circ 1^{*}_{A}\in An(A,B)$. And for any $g\in {\rm{Hom}}(B,A)$, we have $1^{*}_{A}\circ f\in An(B,A)$.
  \item [(4)]
  The compositions defined above are \textit{associative}.
\end{itemize}
\end{definition-proposition}
\begin{remark}
Note that for every $A,B\in \textrm{Ob}(\textbf{Rings})$ and for any $f^{*}\in An(A,B)$, we have $f^{*}\circ 1^{*}_{A}\in \textrm{Hom}(A,B)$. And for any $g^{*}\in An(B,A)$, we have $1^{*}_{A}\circ g^{*}\in \textrm{Hom}(B,A)$. And for the reverse mapping $1^{*}_{A}$, we will suppress the subscript and simply write $1^{*}$ if there is no confusion.
\end{remark}

\begin{proposition}\label{A9}
For the category $\textbf{Rings}$ of rings, we have a one-to-one correspondence of sets ${\rm{Hom}}(A,B)\xrightarrow{\sim} An(A,B)$ for any $A,B\in {\rm{Ob}}(\textbf{Rings})$.
\end{proposition}
\begin{proof}
We can define a function $\textrm{Hom}(A,B)\rightarrow An(A,B),f\mapsto f^{*}=f\circ1^{*}$, and a function $An(A,B)\rightarrow\textrm{Hom}(A,B),f^{*}\mapsto f=f^{*}\circ1^{*}$. Hence we get the one-to-one correspondence $\textrm{Hom}(A,B)\xrightarrow{\sim} An(A,B)$ as desired.
\end{proof}

The following corollary of Proposition \ref{A9} gives \cite{Pursell}, Theorem 7 in the case of rings:
\begin{corollary}\label{A10}
If $A,B\in {\rm{Ob}}(\textbf{Rings})$, then $A,B$ are isomorphic if and only if $A,B$ are anti-isomorphic.
\end{corollary}

With the notation given in Proposition \ref{A12}, we can prove Theorem \ref{A8} in a simple way.
\begin{proof}[Proof of Theorem \ref{A8}]
Let $\phi:A\rightarrow B$ be a homomorphism of rings. Then there is a reverse mapping $1^{*}:A\rightarrow A$ of $A$ and an anti-homomorphism $\varphi:A\rightarrow B$ such that $(\varphi\circ1^{*})(x+y)=\varphi(1^{*}(x+y))=\varphi(x+y)=\varphi(x)+\varphi(y)$ and $(\varphi\circ1^{*})(xy)=\varphi(1^{*}(xy))=\varphi(yx)=\varphi(x)\varphi(y)$ for all $x,y\in A$. So we get $\phi=\varphi\circ1^{*}$ as desired.
\end{proof}

\begin{corollary}
Every homomorphism $\phi:A\rightarrow B$ of rings can be factored into a sequence
$$
A\xrightarrow{1^{*}} A\xrightarrow{\varphi} B,
$$
where $1^{*}$ is a reverse mapping of $A$ and $\varphi$ is an anti-homomorphism.
\end{corollary}

The following proposition gives \cite{Pursell}, Theorem 8 in the case of rings.
\begin{proposition}
If a ring $A$ is not commutative, then ${\rm{Hom.Is}}(A,B)\cap An{\rm{.Is}}(A,B)=\varnothing$ for any $B\in {\rm{Ob}}(\textbf{Rings})$.
\end{proposition}
\begin{proof}
By Proposition \ref{A13}, we have ${\rm{Hom.Is}}(A,B)\cap An{\rm{.Is}}(A,B)=\varnothing$.
\end{proof}

Next, we prove the anti-analogues of the factorization theorem and the homomorphism theorem in the case of rings (see \cite{Grillet}, Chapter III, Theorem 3.5, Theorem 3.6).

\begin{theorem}[Anti-factorization Theorem]\label{A6}
Let $I$ be an ideal of a ring $R$. For every anti-homomorphism $f:R\rightarrow S$ of rings whose kernel contains $I$, there exists a unique anti-homomorphism $\psi:R/I\rightarrow S$ such that the following diagram commutes ($\pi$ is the natural projection):
$$
\xymatrix{
  R \ar[dr]_{f} \ar[r]^{\pi}
                & R/I \ar@{-->}[d]^{\psi}  \\
                & S             }
$$
\end{theorem}
\begin{proof}
By Theorem \ref{A4}, there exists a unique anti-homomorphism $\psi:R/I\rightarrow S$ of abelian groups such that $f=\psi\circ\pi$, i.e. we have $\psi(x+I)=f(x)$ for all $x\in R$. Such an anti-homomorphism of groups is indeed an anti-homomorphism of rings, since
$$
\psi((x+I)(y+I))=\psi(xy+I)=f(xy)=f(y)f(x)=\psi(y+I)\psi(x+I)
$$
for all $x+I,y+I\in R/I$. Moreover, we have $\psi(1+I)=f(1)=1$, which gives us the desired canonical anti-homomorphism of rings.
\end{proof}

\begin{theorem}[Anti-homomorphism Theorem]
If $\varphi:R\rightarrow S$ is an anti-homomorphism of rings, then there is a canonical anti-isomorphism
$$
R/{\rm{Ker}}\varphi\cong {\rm{Im}}\varphi,
$$
that is, there is a unique anti-isomorphism $\xi:R/{\rm{Ker }}\varphi\xrightarrow{\sim} {\rm{Im }}\varphi$ such that the following diagram commutes:
$$
\xymatrix{
  R \ar[d]_{\pi} \ar[r]^{\varphi} & S  \\
  R/{\rm{Ker}}\varphi \ar@{-->}[r]^{\xi} & {\rm{Im}}\varphi \ar[u]^{\imath}  }
$$
($\pi$ is the natural projection and $\imath$ is the inclusion homomorphism).
\end{theorem}
\begin{proof}
By Theorem \ref{A5}, there is a unique anti-isomorphism of groups $\xi:R/\textrm{Ker }\varphi\cong \textrm{Im }\varphi$ such that $\varphi=\pi\circ\xi\circ\imath$, i.e. we have $\xi(x+\textrm{Ker }\varphi)=\varphi(x)$ for all $x\in R$. Then the almost same proof as that of Theorem \ref{A6} shows that $\xi$ is an anti-homomorphism of rings.
\end{proof}

For any $A,B\in \textrm{Ob}(\textbf{Rings})$, we can define addition and multiplication on the sets $\textrm{Hom}(A,B)$ and $An(A,B)$ as follows:
\begin{align*}
(f+g)(x)&=f(x)+g(x) \\
(fg)(x)&=f(x)g(x)
\end{align*}
for all $x\in A$ and $f,g\in\textrm{Hom}(A,B)$.
\begin{align*}
(f^{*}+g^{*})(x)&=f^{*}(x)+g^{*}(x) \\
(f^{*}g^{*})(x)&=f^{*}(x)g^{*}(x)
\end{align*}
for all $x\in A$ and $f^{*},g^{*}\in An(A,B)$.

Then under these two operations, $\textrm{Hom}(A,B)$ and $An(A,B)$ become rings with identity. And we see that $\textrm{Hom}(A,B)$ is isomorphic to $An(A,B)$.
\begin{proposition}
Let $A,B\in {\rm{Ob}}(\textbf{Rings})$. Then there is an isomorphism of rings
$$
{\rm{Hom}}(A,B)\xrightarrow{\sim}An(A,B),\ f\mapsto f^{*}=f\circ1^{*}.
$$
\end{proposition}
\begin{proof}
Let $f,g\in \textrm{Hom}(A,B)$. Then we have
$$
(f+g)^{*}(x)=[(f+g)\circ1^{*}](x)=(f+g)(x)=f(x)+g(x)=f\circ1^{*}(x)+g\circ1^{*}(x)=(f^{*}+g^{*})(x),
$$
which implies that $(f+g)^{*}=f^{*}+g^{*}$. And we have
$$
(fg)^{*}(x)=[(fg)\circ1^{*}](x)=(fg)(x)=f(x)g(x)=(f\circ1^{*})(x)(g\circ1^{*})(x)=(f^{*}g^{*})(x),
$$
which implies that $(fg)^{*}=f^{*}g^{*}$. By Proposition \ref{A9}, we get the desired isomorphism.
\end{proof}

Next, we give a proof of the last proposition in \cite{Dr. Dale T. B.}, whose original proof is not rigorous.

\begin{proposition}
Let $R$ be a ring and $I$ be an ideal of $R$. Then there exists a natural homomorphism $\varphi:An(R,R)\rightarrow An(R,R/I)$.
\end{proposition}
\begin{proof}
If $f\in An(R,R)$, then we define $\phi:R\rightarrow R/I$ by $\phi(x)=f(x)+I$ for every $x\in R$. Clearly, $\phi$ is well-defined, since $f$ is well-defined. And since for any $x,y\in R$, we have $\phi(x+y)=f(x+y)+I=(f(x)+I)+(f(y)+I)=\phi(x)+\phi(y)$ and $\phi(xy)=f(xy)+I=(f(y)+I)(f(x)+I)=\phi(y)\phi(x)$, $\phi$ is an anti-homomorphism and $\phi\in An(R,R/I)$.

Next, we define $\varphi:An(R,R)\rightarrow An(R,R/I)$ by $\varphi(f)=\phi$. If $f_{1}=f_{2}$, then we have $\varphi(f_{1})(x)=f_{1}(x)+I=f_{2}(x)+I=\varphi(f_{2})(x)$ for any $x\in R$, which shows that $\varphi$ is well-defined. And since for any $x\in R$, we have
\begin{align*}
\varphi(f_{1}+f_{2})(x)&=(f_{1}+f_{2})(x)+I \\
&=(f_{1}(x)+I)+(f_{2}+I)  \\
&=\varphi(f_{1})(x)+\varphi(f_{2})(x), \\
\varphi(f_{1}f_{2})(x)&=(f_{1}f_{2})(x)+I \\
&=(f_{1}+I)(f_{2}+I)  \\
&=\varphi(f_{1})(x)\varphi(f_{2})(x).
\end{align*}
This shows that $\varphi$ is a homomorphism.
\end{proof}

Note that we will not define $*$-composition of anti-homomorphisms of rings, leaving it in the next section. It is more convenient to study it in greater generality.

\section{Factorization Categories}\label{B3}

Consider the two categories $\textbf{Groups}$ and $\textbf{Groups-An}$. We can form a new category $\mathscr{G}$ as follows:
$$
\textrm{Ob}(\mathscr{G})=\textrm{Ob}(\textbf{Groups}), \ \textrm{Hom}_{\mathscr{G}}(x,y)=\textrm{Hom}_{\textbf{Groups}}(x,y)\cup\textrm{Hom}_{\textbf{Groups-An}}(x,y),
$$
and the composition is the usual composition. Then $\mathscr{G}$ is the category of groups whose morphisms are homomorphisms and anti-homomorphisms. By Proposition \ref{A3} and \ref{A12}, we think that one can adjoin anti-morphisms to a category in such a way that the ``anti'' phenomenon can be studied in a more general setup.

\begin{definition}
A category $\mathscr{C}$ is said to be equipped with a \textit{factorial structure} $\Sigma$, such that the pair $(\mathscr{C},\Sigma)$ is called a \textit{factorization category}, if it satisfies the following additional conditions:
\begin{itemize}
  \item [(1)]
  To each pair of objects \textit{A,B} $\in \textrm{Ob}(\mathscr{C})$, in addition to the set $\textrm{Hom}(A,B)$ of morphisms, there is a set $An(A,B)$ of \textit{anti-morphisms from $A$ to $B$}.
  \item [(2)]
  To each triple of objects \textit{A,B,C} $\in \textrm{Ob}(\mathscr{C})$, there are \textit{laws of compositions}
   \begin{align*}
  &An(A,B) \times An(B,C) \rightarrow \textrm{Hom}(A,C), \ \ (f',g')\mapsto f'\circ g' \\
  &An(A,B) \times \textrm{Hom}(B,C) \rightarrow An(A,C), \ \ (f',g)\mapsto f'\circ g  \\
  &\textrm{Hom}(A,B) \times An(B,C) \rightarrow An(A,C), \ \ (f,g')\mapsto f\circ g'
  \end{align*}
  \item [(3)]
  For every $A,B\in \textrm{Ob}(\mathscr{C})$, there exists a \textit{reverse morphism} $1^{*}_{A}\in An(A,A)$ such that for any morphism $f\in \textrm{Hom}(A,B)$, we have
  $f\circ 1^{*}_{A}\in An(A,B)$. And for any $g\in \textrm{Hom}(B,A)$, we have $1^{*}_{A}\circ f\in An(B,A)$.
  \item [(4)]
  The compositions defined above are \textit{associative}.
\end{itemize}
\end{definition}
\begin{remark}
Note that for every $A,B\in \textrm{Ob}(\mathscr{C})$ and for any $f^{*}\in An(A,B)$, we have $f^{*}\circ 1^{*}_{A}\in \textrm{Hom}(A,B)$. And for any $g^{*}\in An(B,A)$, we have $1^{*}_{A}\circ g^{*}\in \textrm{Hom}(B,A)$. Moreover, for $f\in\textrm{Hom}(A,B)$, we have $f\circ1^{*}_{A}=1^{*}_{B}\circ f$. By abuse of notation, we will simply write $\mathscr{C}$ instead of $(\mathscr{C},\Sigma)$ for a factorization category, when no ambiguities arise. For the reverse morphism $1^{*}_{A}$, we will suppress the subscript and simply write $1^{*}$ if there is no confusion.
\end{remark}

However, sometimes, one should assure that every morphism has kernel and cokernel, etc. Hence the notions of some special kinds of categories provide us with a good framework.

\begin{definition}\

\begin{itemize}
\item[(1)]
A factorization category is \textit{preadditive/additive/abelian} if its underlying category is preadditive/additive/abelian and to each triple of objects $A,B,C\in\textrm{Ob}(\mathscr{C})$, the laws of compositions
\begin{align*}
&An(A,B) \times \textrm{Hom}(B,C) \rightarrow An(A,C)\\
&\textrm{Hom}(A,B) \times An(B,C) \rightarrow An(A,C)
\end{align*}
are linear in the group $\textrm{Hom}(B,C)$ and $\textrm{Hom}(A,B)$.
\item[(2)]
A \textit{factorization category with zero morphisms} is a factorization category whose underlying category is a category with zero morphisms.
\item[(3)]
A \textit{factorization category with products} is a factorization category whose underlying category is a category that admits products.
\end{itemize}
\end{definition}

\begin{example}
The category $\textbf{Groups}$ of groups and the category $\textbf{Rings}$ of rings are factorization categories, whose morphisms are homomorphisms and anti-morphisms are anti-homomorphisms.
\end{example}
\begin{example}
The category $\textbf{Ab}$ of abelian groups is an abelian factorization category, whose morphisms are homomorphisms and anti-morphisms are anti-homomorphisms.
\end{example}
\begin{example}
It is easy to check that the category $\textrm{\textbf{Vec}}_{\mathbb{C}}$ of complex vector spaces is an abelian factorization category, whose anti-morphisms are anti-linear maps, and reverse morphisms are the conjugate maps $\psi: V\rightarrow V$ defined by $\psi(\lambda x)=\overline{\lambda}x$ for $\lambda\in\mathbb{C}$ and $x\in V$.
\end{example}

The following proposition generalizes our results in Proposition \ref{A1} and \ref{A9}.
\begin{proposition}\label{A17}
For any factorization category $\mathscr{C}$, we have a one-to-one correspondence of sets ${\rm{Hom}}(A,B)\xrightarrow{\sim} An(A,B)$ for any $A,B\in {\rm{Ob}}(\mathscr{C})$. Moreover, if $\mathscr{C}$ is a preadditive factorization category, $An(A,B)$ induces the abelian group law from ${\rm{Hom}}(A,B)$ so that we have a group isomorphism ${\rm{Hom}}(A,B)\cong An(A,B)$.
\end{proposition}
\begin{proof}
In fact, we can define a function
$$
\textrm{Hom}(A,B)\rightarrow An(A,B),\ f\mapsto f^{*}=f\circ1^{*},
$$
and a function
$$
An(A,B)\rightarrow\textrm{Hom}(A,B),\ f^{*}\mapsto f=f^{*}\circ1^{*}.
$$
Hence we get the one-to-one correspondence $\textrm{Hom}(A,B)\xrightarrow{\sim} An(A,B)$ as desired.

By definition, we have $(f+g)\circ1^{*}=f\circ1^{*}+g\circ1^{*}$, which shows that $(f+g)^{*}=f^{*}+g^{*}$.
\end{proof}
\begin{remark}
For any morphism $f$, $f^{*}$ will stand for $f\circ1^{*}=1^{*}\circ f$, which is called the \textit{corresponding anti-morphism of} $f$. Since every anti-morphism $f'$ can be written as $f\circ1^{*}$ or $1^{*}\circ f$ for some morphism, we can write $f^{*}$ for any anti-morphism as well. Similarly, for any anti-morphism $f^{*}$, the morphism $f=f^{*}\circ1^{*}=1^{*}\circ f^{*}$ is called the \textit{corresponding morphism of} $f^{*}$. And clearly, we have $f^{**}=f$.
\end{remark}

Since the composition of two anti-morphisms is a morphism, it is natural to think if any morphism can be expressed as a composition of anti-morphisms. This leads to the following theorem.
\begin{theorem}\label{A26}
Every morphism can be expressible as a composite of anti-morphisms.
\end{theorem}
\begin{proof}
Let $\mathscr{C}$ be a factorization category and let $f:x\rightarrow y$ be any morphism in $\mathscr{C}$. Then there exist a reverse morphism $1^{*}:x\rightarrow x$ and an anti-morphism $f^{*}:x\rightarrow y$ in $\mathscr{C}$ such that $f=f^{*}\circ1^{*}$.
\end{proof}

\begin{corollary}
Every morphism $f:x\rightarrow y$ can be factored into a sequence
$$
x\xrightarrow{1^{*}} x\xrightarrow{f^{*}} y,
$$
where $1^{*}$ is the reverse morphism and $\varphi$ is an anti-morphism.
\end{corollary}

Note that by Theorem \ref{A26}, we can define an equivalence relation. First, consider the following solid diagram
$$
\xymatrix{
  x \ar[d]_{g_{1}^{*}} \ar[r]^{f_{1}^{*}}\ar@{-->}[dr]^{f} & y \ar[d]^{f_{2}^{*}} \\
  z \ar[r]^{g_{2}^{*}} & w   }
$$
such that $f$ is a morphism and $(f_{1}^{*},f_{2}^{*}),(g_{1}^{*},g_{2}^{*})$ are two pairs of composable anti-morphisms. Define $(f_{1}^{*},f_{2}^{*})\sim(g_{1}^{*},g_{2}^{*})$ if and only if $f=f_{2}^{*}\circ f_{1}^{*}=g_{2}^{*}\circ g_{1}^{*}$, i.e. if and only if the above diagram is commutative. In other words, if a morphism $f$ is factored into two pairs of composable anti-morphisms $(f_{1}^{*},f_{2}^{*}),(g_{1}^{*},g_{2}^{*})$, then we have $(f_{1}^{*},f_{2}^{*})\sim(g_{1}^{*},g_{2}^{*})$. Next, we show that we have defined an equivalence relation. First, the relation is clearly reflexive and symmetric. To prove transitivity, if $(h_{1}^{*},h_{2}^{*})$ is another pair of composable anti-morphisms such that $(g_{1}^{*},g_{2}^{*})\sim(h_{1}^{*},h_{2}^{*})$.
$$
\xymatrix{
  x \ar[d]_{g_{1}^{*}} \ar[r]^{h_{1}^{*}} & p \ar[d]^{h_{2}^{*}} \\
  z \ar[r]^{g_{2}^{*}} & w   }
$$
Then we have $h_{2}^{*}\circ h_{1}^{*}=g_{2}^{*}\circ g_{1}^{*}=f_{2}^{*}\circ f_{1}^{*}$, which shows that $(f_{1}^{*},f_{2}^{*})\sim(h_{1}^{*},h_{2}^{*})$. Thus we have defined an equivalence relation on anti-morphisms. We denote by $[(f_{1}^{*},f_{2}^{*})]$ the \textit{equivalence class} of the pair $(f_{1}^{*},f_{2}^{*})$. We use the notation $[f]$ to indicate the \textit{set of all equivalence classes of pairs of anti-morphisms whose compositions are the morphism} $f$. And we denote the \textit{set of all equivalence classes of all pairs of anti-morphisms of the form} $(x\rightarrow y,y\rightarrow w)$ or $(x\rightarrow z,z\rightarrow w)$ by $[An(x,y),An(y,w)]$ or $[An(x,z),An(z,w)]$. Then we have the following proposition.
\begin{proposition}
Let $\mathscr{C}$ be a factorization category. For any $A,B,C\in {\rm{Ob}}(\mathscr{C})$, we have a law of factorization
$$
{\rm{Hom}}(A,C) \rightarrow [An(A,B),An(B,C)], \ \ f\mapsto [f].
$$
\end{proposition}

\begin{definition}
Let $\mathscr{C}$ be a factorization category and let $f^{*}:x\rightarrow y$ be an anti-morphism in $\mathscr{C}$. We say that $f^{*}$ is an \textit{anti-isomorphism} if there exists an anti-morphism $g^{*}:y\rightarrow x$ in $\mathscr{C}$ such that $f^{*}\circ g^{*}=1_{y}$ and $g^{*}\circ f^{*}=1_{x}$. Then $x$ is said to be \textit{anti-isomorphic} to $y$, which is denoted by $x\cong y$. We call $g^{*}$ the \textit{anti-inverse} of $f^{*}$. Given $x,y\in\textrm{Ob}(\mathscr{C})$, we denote the set of isomorphisms by ${\rm{Hom.Is}}(x,y)$ and the set of anti-isomorphisms by $An{\rm{.Is}}(x,y)$.
\end{definition}

\begin{proposition}\label{A16}
Let $\mathscr{C}$ be a factorization category and let $f:x\rightarrow y$ be a morphism in $\mathscr{C}$. Then $f$ is an isomorphism if and only if $f^{*}$ is an anti-isomorphism. If $f^{-1}$ is its inverse, then $f^{-1*}$ is its anti-inverse.
\end{proposition}
\begin{proof}
If $f$ is an isomorphism, then $f\circ f^{-1}=1$ gives
$$
f^{*}\circ f^{-1*}=(1^{*}\circ f)\circ (f^{-1}\circ1^{*})=1^{*}\circ (f\circ f^{-1})\circ1^{*}=1^{*}\circ 1\circ1^{*}=1
$$
and
$$
f^{-1*}\circ f^{*}=(1^{*}\circ f^{-1})\circ (f\circ1^{*})=1^{*}\circ (f^{-1}\circ f)\circ1^{*}=1^{*}\circ 1\circ1^{*}=1
$$
which show that $f^{-1*}$ is an anti-inverse of $f^{*}$, and thus $f^{*}$ is an anti-isomorphism.

Conversely, if $f^{*}$ is an anti-isomorphism, then we see that $f$ is an isomorphism in a similar manner.
\end{proof}

In the sequel, we will denote the anti-inverse of $f^{*}$ by $f^{-1*}$. Proposition \ref{A16} readily derives the following corollaries.

\begin{corollary}\label{A30}
Let $\mathscr{C}$ be a factorization category and let $x,y\in{\rm{Ob}}(\mathscr{C})$. Then $x$ is anti-isomorphic to $y$ if and only if $x$ is isomorphic to $y$.
\end{corollary}

\begin{corollary}
For any factorization category $\mathscr{C}$, we have a one-to-one correspondence of sets ${\rm{Hom.Is}}(x,y)\xrightarrow{\sim} An{\rm{.Is}}(x,y)$ for any $x,y\in {\rm{Ob}}(\mathscr{C})$.
\end{corollary}

We want to construct a new category from a factorization category. The following theorem specifies the construction process.

\begin{theorem}\label{A22}
Every factorization category $\mathscr{C}$ gives rise to a category $\widetilde{\mathscr{C}}$, which is called {\rm{the category associated to}} $\mathscr{C}$.
\end{theorem}
\begin{proof}
We construct a category $\widetilde{\mathscr{C}}$ from $\mathscr{C}$. First, we let $\textrm{Ob}(\widetilde{\mathscr{C}})=\textrm{Ob}(\mathscr{C})$. Then for $A,B\in\textrm{Ob}(\widetilde{\mathscr{C}})$, we define morphisms in $\widetilde{\mathscr{C}}$ by
$$
\textrm{Hom}_{\widetilde{\mathscr{C}}}(A,B)=An_{\mathscr{C}}(A,B)\cup \textrm{Hom}_{\mathscr{C}}(A,B).
$$
By definition, $\widetilde{\mathscr{C}}$ is a category.
\end{proof}

It would be a nice result if the category $\widetilde{\mathscr{C}}$ associated to an abelian factorization category $\mathscr{C}$ is an abelian category. Unfortunately, the following theorem shows that $\widetilde{\mathscr{C}}$ is not abelian.
\begin{theorem}\label{A24}
Let $\mathscr{C}$ be an abelian factorization category. Then the category $\widetilde{\mathscr{C}}$ associated to $\mathscr{C}$ is an almost abelian category, i.e. it satisfies the conditions of abelian category except being an additive category.
\end{theorem}
\begin{proof}
First, since $\mathscr{C}$ is an abelian category, all kernels and cokernels of morphisms exist, so do anti-morphisms. And for any morphism $f$ in $\mathscr{C}$, there is a canonical isomorphism
$$
\textrm{Coim}(f)\xrightarrow{\sim}\textrm{Im}(f),
$$
so for any anti-morphism $f^{*}=f\circ1^{*}$, we have a canonical anti-isomorphism
$$
\textrm{Coim}(f^{*})\xrightarrow{\sim}\textrm{Im}(f^{*}).
$$
However, for any $x,y\in\widetilde{\mathscr{C}}$, the set $\textrm{Hom}_{\widetilde{\mathscr{C}}}(x,y)$ defined as a union of $An_{\mathscr{C}}(x,y)$ and $\textrm{Hom}_{\mathscr{C}}(x,y)$ can not form a group, since $An_{\mathscr{C}}(x,y)\nsubseteq\textrm{Hom}_{\mathscr{C}}(x,y)$ and $An_{\mathscr{C}}(x,y)\nsupseteq\textrm{Hom}_{\mathscr{C}}(x,y)$.
\end{proof}

Next, we would like to study kernels, cokernels, monomorphisms, epimorphisms, etc., in a factorization category $\mathscr{C}$ (see \cite{Stack Project}, Definition 12.3.9 for expositions in a preadditive category). Let $f:x\rightarrow y$ be a morphism in $\mathscr{C}$ such that the kernel and cokernel of $f$ exist. Consider the kernel $\tau:\textrm{Ker}(f)\rightarrow x$ of $f$, we have $f\tau=0$ and if there is $\mu:z\rightarrow x$ with $f\mu=0$, there exists a unique $\tau_{0}:z\rightarrow\textrm{Ker}(f)$ such that $f\tau\tau_{0}=0$. Then if we consider an anti-morphism $f^{*}=1^{*}f$, we still have $f^{*}\tau=0$ and if there is $\mu:z\rightarrow x$ with $f^{*}\mu=0$, there exists a unique $\tau_{0}:z\rightarrow\textrm{Ker}(f)$ such that $f^{*}\tau\tau_{0}=0$. This is similar for cokernels.

\begin{definition}
Let $\mathscr{C}$ be a factorization category with zero morphisms and let $f^{*}:x\rightarrow y$ be an anti-morphism in $\mathscr{C}$.
\begin{itemize}
  \item [(1)]
  We define the \textit{kernel} of $f^{*}$ to be the kernel of $f$, denoted by $\textrm{Ker}(f^{*})\rightarrow x$.
  \item [(2)]
  We define the \textit{cokernel} of $f^{*}$ to be the cokernel of $f$, denoted by $y\rightarrow \textrm{Coker}(f^{*})$.
  \item [(3)]
  We define the \textit{coimage} of $f^{*}$ to be the coimage of $f$, denoted by $x\rightarrow \textrm{Coim}(f^{*})$.
  \item [(4)]
  We define the \textit{image} of $f^{*}$ to be the image of $f$, denoted by $\textrm{Im}(f^{*})\rightarrow y$.
\end{itemize}

If $\mathscr{C}$ is an abelian factorization category, then
\begin{itemize}
\item [(I)]
We say that $f^{*}$ is \textit{injective} if $\textrm{Ker}(f^{*})=0$. And we say that $f^{*}$ is \textit{surjective} if $\textrm{Coker}(f^{*})=0$.
\item [(II)]
If $f^{*}$ is injective, then $x$ is a \textit{subobject} of $y$ and we use the notation $x\subset y$ to indicate this. Moreover, we denote by $y/x$ the object
$$
\textrm{Coker}(x\rightarrow y).
$$
\item [(III)]
If $f^{*}$ is surjective, then we say that $y$ is a \textit{quotient} of $x$.
\end{itemize}
\end{definition}

Recall that a morphism $u:A\rightarrow B$ in a category $\mathscr{C}$ is a \textit{monomorphism} if the function
$$
\textrm{Hom}(C,A)\rightarrow \textrm{Hom}(C,B),\ \ v\mapsto u\circ v
$$
is injective for all $C\in \textrm{Ob}(\mathscr{C})$. Then for an anti-morphism $u^{*}$, since $1^{*}uv_{1}=1^{*}uv_{2}\Rightarrow uv_{1}=uv_{2}$, we still have $u^{*}v_{1}=u^{*}v_{2}\Rightarrow v_{1}=v_{2}$. This is similar for epimorphisms.

\begin{definition}
Let $\mathscr{C}$ be a factorization category and let $f^{*}:x\rightarrow y$ be an anti-morphism in $\mathscr{C}$. Then $f^{*}$ is an \textit{anti-monomorphism} if the function
$$
\textrm{Hom}(z,x)\rightarrow An(z,y),\ \ v\mapsto f^{*}\circ v
$$
is injective for all $z\in \textrm{Ob}(\mathscr{C})$. And we say that $f^{*}$ is an \textit{anti-epimorphism} if the function
$$
\textrm{Hom}(y,z)\rightarrow An(x,z),\ \ v\mapsto v\circ f^{*}
$$
is injective for all $z\in \textrm{Ob}(\mathscr{C})$.
\end{definition}

Next, we give some elementary results of anti-morphisms, which show that anti-morphisms inherit some properties of morphisms in a preadditive category and an abelian category (see \cite{Stack Project}, Lemma 12.5.4 and Lemma 12.3.12).

\begin{lemma}
Let $f^{*}:x\rightarrow y$ be an anti-morphism in a preadditive factorization category. If the kernel, cokernel, coimage, and image exist, then $f^{*}$ factors uniquely as the following sequence:
$$
x\rightarrow {\rm{Coim}}(f^{*})\rightarrow {\rm{Im}}(f^{*})\rightarrow y.
$$
\end{lemma}
\begin{proof}
The composition $\textrm{Ker}(f^{*})\rightarrow x\xrightarrow{f^{*}}y$ is zero, so there is a canonical map $\textrm{Coim}(f^{*})\rightarrow y$.
$$
\xymatrix{
  \textrm{Ker}(f^{*})  \ar[r]^{} & x \ar[dr]_{} \ar[r]^{f} & y  \ar[r]^{1^{*}} & y \\
    &   & \textrm{Coim}(f^{*})\ar[u]_{} \ar@{-->}[ur]_{} &    }
$$
The composition $x\xrightarrow{f^{*}}y\rightarrow\textrm{Coker}(f^{*})$ is zero, and we have a factorization:
$$
\xymatrix{
   & x \ar[dr]_{} \ar[r]^{f^{*}} & y \ar[r]^{}  & \textrm{Coker}(f^{*})  \\
  &  & \textrm{Coim}(f^{*}) \ar[u]^{} &    }
$$
Since $x\rightarrow\textrm{Coim}(f^{*})$ is surjective, the composition $\textrm{Coim}(f^{*})\rightarrow y\rightarrow \textrm{Coker}(f^{*})$ is zero. Hence, $\textrm{Coim}(f^{*})\rightarrow y$ factors through $\textrm{Im}(f^{*})\rightarrow y$, which gives us the desired map.
\end{proof}

\begin{lemma}
Let $f^{*}:x\rightarrow y$ be an anti-morphism in an abelian factorization category $\mathscr{C}$. Then
\begin{itemize}
  \item [(1)]
  $f^{*}$ is injective if and only if $f^{*}$ is an anti-monomorphism.
  \item [(2)]
  $f^{*}$ is surjective if and only if $f^{*}$ is an anti-epimorphism.
\end{itemize}
\end{lemma}
\begin{proof}
Note that $\textrm{Ker}(f^{*})$ is an object representing the functor sending an object $z\in \textrm{Ob}(\mathscr{C})$ to the set $\textrm{Ker}(\textrm{Hom}(z,x)\rightarrow An(z,y))$. In fact, for $z\in \textrm{Ob}(\mathscr{C})$, the following composition of maps
$$
\textrm{Hom}(z,\textrm{Ker}(f^{*}))\rightarrow \textrm{Hom}(z,x)\rightarrow \textrm{Hom}(z,y)\rightarrow An(z,y)
$$
is zero. And we have $\textrm{Hom}(z,\textrm{Ker}(f^{*}))=\textrm{Ker}(\textrm{Hom}(z,x)\rightarrow An(z,y))$.

Then $\textrm{Ker}(f^{*})=0$ if and only if the functions $\textrm{Hom}(z,x)\rightarrow An(z,y)$ are injective for all $z\in\textrm{Ob}(\mathscr{C})$ if and only if $f^{*}$ is an anti-monomorphism. Another case is similar.
\end{proof}

In the following, we will prove anti-homomorphism theorem, anti-factorization theorem, and second anti-isomorphism theorem in the case of abelian factorization category.
\begin{theorem}[Generalized anti-homomorphism theorem]
Let $\mathscr{C}$ be an abelian factorization category and let $f^{*}:x\rightarrow y$ be an anti-morphism in $\mathscr{C}$. Then we have an anti-isomorphism $x/{\rm{Ker}}(f^{*})\cong{\rm{Im}}(f^{*})$. In particular, if $f^{*}$ is injective, then $x\cong{\rm{Im}}(f^{*})$. And if $f^{*}$ is surjective, then $x/{\rm{Ker}}(f^{*})\cong y$.
\end{theorem}
\begin{proof}
By Theorem \ref{A24}, we have a canonical anti-isomorphism $\textrm{Coim}(f^{*})\cong\textrm{Im}(f^{*})$, which is in fact the anti-isomorphism $x/{\rm{Ker}}(f^{*})\cong{\rm{Im}}(f^{*})$.
\end{proof}

\begin{theorem}[Generalized anti-factorization theorem]
Let $\mathscr{C}$ be an abelian factorization category and let $f^{*}:x\rightarrow y$ be an anti-morphism in $\mathscr{C}$. If $\mu:z\rightarrow x$ is an injective morphism such that $f\mu=0$, then there exists a unique anti-morphism $\psi^{*}:x/z\rightarrow y$ such that the following diagram commutes ($\pi$ is the canonical map):
$$
\xymatrix{
  x \ar[dr]_{f^{*}} \ar[r]^{\pi}
                & x/z \ar@{-->}[d]^{\exists!\psi^{*}}  \\
                & y             }
$$
\end{theorem}
\begin{proof}
First, we consider the factorization of morphism $f:x\rightarrow y$. The composition $z\rightarrow x\rightarrow y$ is zero, so there is a canonical map $\psi:x/z\rightarrow y$ making the triangle commutes (note that $\pi$ is the cokernel of $\mu$):
$$
\xymatrix{
  z  \ar[r]^{\mu} & x \ar[dr]_{\pi} \ar[r]^{f} & y  \\
    &   &  x/z\ar@{-->}[u]_{\psi}  }
$$
Then the following calculation gives us the desired result
\begin{align*}
1^{*}\circ f&=1^{*}\circ\psi\circ\pi \\
f^{*}&=\psi^{*}\circ\pi.
\end{align*}
\end{proof}

\begin{theorem}[Generalized second anti-isomorphism theorem]
Let $\mathscr{C}$ be an abelian factorization category and let $A,B,C\in{\rm{Ob}}(\mathscr{C})$ such that $C\subset B\subset A$. Then there is a unique anti-isomorphism
$$
(A/C)/(B/C)\cong A/B.
$$
\end{theorem}
\begin{proof}
Since $\mathscr{C}$ is an abelian category, we have a unique isomorphism $(A/C)/(B/C)\cong A/B$. By Proposition \ref{A16} and Corollary \ref{A30}, we get the unique anti-isomorphism as desired.
\end{proof}

Fix a factorization category $\mathscr{C}$. Next, we define the \textit{$*$-composition} of anti-morphisms as follows:
$$
f^{*}*g^{*}:=f^{*}\circ g^{*}\circ1^{*},
$$
when $f\in An(x,y),g\in An(y,z)$ for $x,y,z\in\textrm{Ob}(\mathscr{C})$. It is easy to see that the $*$-composition is well-defined, and the $*$-composition of two anti-morphisms is an anti-morphism. We have the following propositions, which generalize our results in $\textbf{Groups}$.

\begin{proposition}\label{A18}
The $*$-composition of anti-morphisms is associative.
\end{proposition}
\begin{proof}
Let $f^{*},g^{*},h^{*}$ be three anti-morphisms in $\mathscr{C}$. When the compositions make sense, we have
$$
(f^{*}*g^{*})*h^{*}=(f^{*}\circ g^{*}\circ1^{*})\circ h^{*}\circ1^{*}=f^{*}\circ g^{*}\circ h^{*}
$$
and
$$
f^{*}*(g^{*}*h^{*})=f^{*}\circ (g^{*}\circ h^{*}\circ1^{*})\circ1^{*}=f^{*}\circ g^{*}\circ h^{*}.
$$
Hence $(f^{*}*g^{*})*h^{*}=f^{*}*(g^{*}*h^{*})$.
\end{proof}

\begin{proposition}\label{A19}
If $f^{*}$ is an anti-isomorphism and $f^{-1*}$ is its anti-inverse, then $f^{*}*f^{-1*}=f^{-1*}*f^{*}=1^{*}$.
\end{proposition}
\begin{proof}
By definition, we have
$$
f^{*}*f^{-1*}=(f^{*}\circ f^{-1*})\circ1^{*}=1\circ1^{*}=1^{*},
$$
and
$$
f^{-1*}*f^{*}=(f^{-1*}\circ f^{*})\circ1^{*}=1\circ1^{*}=1^{*}.
$$
\end{proof}

\begin{proposition}\label{A20}
If $f^{*}$ is an anti-homomorphism, then $f^{*}*1^{*}=1^{*}*f^{*}=f^{*}$.
\end{proposition}
\begin{proof}
By definition, we have
$$
f^{*}*1^{*}=(f^{*}\circ1^{*})\circ1^{*}=f^{*},
$$
and
$$
1^{*}*f^{*}=(1^{*}\circ f^{*})\circ1^{*}=f^{*}.
$$
\end{proof}

By means of $*$-composition, we can construct a new category $\mathscr{C}^{An}$ in terms of anti-morphisms as follows:
$$
\textrm{Ob}(\mathscr{C}^{An}):=\textrm{Ob}(\mathscr{C}), \ \ \textrm{Hom}_{\mathscr{C}^{An}}(x,y):=An(x,y),
$$
and
$$
\textrm{Hom}_{\mathscr{C}^{An}}(x,y)\times \textrm{Hom}_{\mathscr{C}^{An}}(y,z)\rightarrow \textrm{Hom}_{\mathscr{C}^{An}}(x,z), \ \ (f^{*},g^{*})\mapsto f^{*}*g^{*}.
$$
i.e. the composition is defined as the $*$-composition.
\begin{theorem}\label{A23}
If $\mathscr{C}$ is an abelian factorization category, the definition above yields an abelian category $\mathscr{C}^{An}$.
\end{theorem}
We just prove a special case here and defer the proof below.
\begin{lemma}\label{A27}
If $\mathscr{C}$ is an abelian factorization category, the definition above yields a preadditive category $\mathscr{C}^{An}$.
\end{lemma}
\begin{proof}
Proposition \ref{A18}, \ref{A19}, and \ref{A20} show that $\mathscr{C}^{An}$ is a category.

For any $x,y\in\textrm{Ob}(\mathscr{C}^{An})$, $An(x,y)$ is an abelian group, by Proposition \ref{A17}. Next, we show that the composition is bilinear. Let $f^{*},g^{*},h^{*}$ be anti-morphisms in $\mathscr{C}^{An}$. When the compositions make sense, we have
\begin{align*}
h^{*}*(f+g)^{*}&=h^{*}\circ(f+g)^{*}\circ1^{*} \\
&=h^{*}\circ(f+g) \\
&=h^{*}\circ f+h^{*}\circ g \\
&=h^{*}*f^{*}+h^{*}*g^{*},
\end{align*}
and
\begin{align*}
(f+g)^{*}*h^{*}&=(f+g)^{*}\circ h^{*}\circ1^{*} \\
&=(f+g)\circ h^{*} \\
&=f\circ h^{*}+g\circ h^{*} \\
&=f^{*}*h^{*}+g^{*}*h^{*}.
\end{align*}
So $\mathscr{C}^{An}$ is a preadditive category.
\end{proof}

We want to clarify the relation between $\mathscr{C}$ and $\mathscr{C}^{An}$, so one need to define functors between $\mathscr{C}$ and $\mathscr{C}^{An}$.

First, we can define a functor
$$
F^{An}:\mathscr{C}\rightarrow \mathscr{C}^{An}
$$
which assigns to each object of $\mathscr{C}$ the same object, and to every morphism $f$, it associates the corresponding anti-morphism $f^{*}$.

And we can define a functor
$$
G^{An}:\mathscr{C}^{An}\rightarrow\mathscr{C}
$$
which assigns to each object of $\mathscr{C}^{An}$ the same object, and to each anti-morphism $f^{*}$, it associates the corresponding morphism $f$.

\begin{theorem}\label{A21}
The categories $\mathscr{C}$ and $\mathscr{C}^{An}$ are equivalent, i.e. $\mathscr{C}\cong\mathscr{C}^{An}$.
\end{theorem}
\begin{proof}
Clearly every object $x$ of $\mathscr{C}^{An}$ is isomorphic to $F^{An}(x)$. And by Proposition \ref{A17}, we see that $F^{An}$ is fully faithful. So $F^{An}$ is an equivalence of categories.
\end{proof}
\begin{proof}[Proof of Theorem \ref{A23}]
By Theorem \ref{A21}, if $\mathscr{C}$ is an abelian category, the equivalence of categories implies that $\mathscr{C}^{An}$ is also an abelian category.
\end{proof}

\begin{definition}
Consider the situation above, the category $\mathscr{C}^{An}$ so constructed is called the \textit{anti-category associated to} $\mathscr{C}$. If the factorization category $\mathscr{C}$ is abelian, the category $\mathscr{C}^{An}$ so constructed is called the \textit{abelian anti-category associated to} $\mathscr{C}$.
\end{definition}

If the factorization category $\mathscr{C}$ is abelian, then using $*$-composition, we can define a bifunctor $An(-,-)$ like $\textrm{Hom}(-,-)$ as follows:
\begin{itemize}
\item[(1)]
For any $x,y\in\textrm{Ob}(\mathscr{C})$, $An(x,y)$ is the group of anti-morphisms from $x$ to $y$.
\item[(2)]
If $y\in\textrm{Ob}(\mathscr{C})$ and $f^{*}:x\rightarrow y$ is an anti-morphism in $\mathscr{C}$, then $An(f^{*},z)$ is defined by
$$
An(f^{*},z):An(y,z)\rightarrow An(x,z), \ \ g^{*}\mapsto g^{*}*f^{*}.
$$
Similarly, we define $An(z,f^{*})$ by
$$
An(z,f^{*}):An(z,x)\rightarrow An(z,y), \ \ g^{*}\mapsto f^{*}*g^{*}.
$$
\end{itemize}

\begin{theorem}
$An(-,-)$ is a bifunctor from the abelian associated anti-category $\mathscr{C}^{An}$ to the category $\textbf{Ab}$ of abelian groups. It is contravariant in the first variable, and covariant in the second variable.
\end{theorem}
\begin{proof}
By Proposition \ref{A17}, Lemma \ref{A27}, and the above definitions.
\end{proof}

\begin{theorem}
The functors $An(-,-)$ and ${\rm{Hom}}(-,-)$ are isomorphic.
\end{theorem}
\begin{proof}
It is easy to check using Proposition \ref{A17}.
\end{proof}

Since a factorization category induces a category, it is natural to think if one can obtain a factorization category from a given category. However, we are particularly curious about how to generate a factorization category through two equivalent categories. The process is as follows:

\begin{construction}\label{A34}
Let $\mathscr{C}$ be a category and let $\mathscr{D}$ be a category equivalent to $\mathscr{C}$ such that ${\rm{Ob}}(\mathscr{C})={\rm{Ob}}(\mathscr{D})$. If the following conditions are satisfied:
\begin{itemize}
  \item [(1)]
  To each triple of objects \textit{A,B,C} $\in {\rm{Ob}}(\mathscr{C})$, there are laws of compositions
  \begin{align*}
  &{\rm{Hom}}_{\mathscr{D}}(A,B) \times {\rm{Hom}}_{\mathscr{D}}(B,C) \rightarrow {\rm{Hom_{\mathscr{C}}}}(A,C),\ \ (f',g')\mapsto f'\odot g' \\
  &{\rm{Hom}}_{\mathscr{D}}(A,B) \times {\rm{Hom_{\mathscr{C}}}}(B,C) \rightarrow {\rm{Hom}}_{\mathscr{D}}(A,C),\ \ (f',g)\mapsto f'\odot g  \\
  &{\rm{Hom_{\mathscr{C}}}}(A,B) \times {\rm{Hom}}_{\mathscr{D}}(B,C) \rightarrow {\rm{Hom}}_{\mathscr{D}}(A,C),\ \ (f,g')\mapsto f\odot g'
  \end{align*}
  \item [(2)]
  For every $A,B\in {\rm{Ob}}(\mathscr{C})$ and the identity morphism $1'_{A}\in {\rm{Hom}}_{\mathscr{D}}(A,A)$, we have $f\odot 1'_{A}\in {\rm{Hom}}_{\mathscr{D}'}(A,B)$ for any morphism $f\in {\rm{Hom_{\mathscr{D}}}}(A,B)$. And for any $g\in {\rm{Hom_{\mathscr{D}}}}(B,A)$, we have $1'_{A}\odot f\in {\rm{Hom}}_{\mathscr{D}'}(B,A)$.
  \item[(3)]
  The compositions defined above are associative.
\end{itemize}
Then we obtain a factorization category $\overline{\mathscr{C}}=(\mathscr{C},\Sigma)=\underline{(\mathscr{C},\mathscr{D})}$ with ${\rm{Ob}}(\overline{\mathscr{C}})={\rm{Ob}}(\mathscr{C})$, ${\rm{Hom_{\overline{\mathscr{C}}}}}(A,B)={\rm{Hom}}_{\mathscr{C}}(A,B)$, and $An(A,B)={\rm{Hom}}_{\mathscr{D}}(A,B)$.
\end{construction}

\begin{definition}
Let $\mathscr{C}$ be a category. The factorization category $\overline{\mathscr{C}}$ constructed above is called the \textit{factorization category associated to} $\mathscr{C}$. And the pair $(\mathscr{C},\mathscr{D})$ of two equivalent categories as above is called the \textit{generator of the associated factorization category}.
\end{definition}

So a factorization category induces two equivalent categories, and a category together with an equivalent category induce a factorization category. However, since a category may associate to different factorization categories, to avoid ambiguity, we have to impose the following axiom for factorization categories.

\begin{axiom}\label{A29}
Every category can be equipped with a unique factorial structure making it a factorization category.
\end{axiom}

In other words, consider the situation in Construction \ref{A34}. If $\mathscr{D}'$ is another category equivalent to $\mathscr{C}$ with $\textrm{Ob}(\mathscr{D}')=\textrm{Ob}(\mathscr{C})$, then we have $\underline{(\mathscr{C},\mathscr{D})}=\underline{(\mathscr{C},\mathscr{D}')}=\overline{\mathscr{C}}$.

Moreover, we can generate a preadditive factorization category through two equivalent preadditive categories. The process is as follows:
\begin{construction}\label{A34}
Let $\mathscr{C}$ be a preadditive category and let $\mathscr{D}$ be a category equivalent to $\mathscr{C}$ such that ${\rm{Ob}}(\mathscr{C})={\rm{Ob}}(\mathscr{D})$. If the following conditions are satisfied:
\begin{itemize}
  \item [(1)]
  To each triple of objects \textit{A,B,C} $\in {\rm{Ob}}(\mathscr{C})$, there are laws of compositions that are bilinear
  \begin{align*}
  &{\rm{Hom}}_{\mathscr{D}}(A,B) \times {\rm{Hom}}_{\mathscr{D}}(B,C) \rightarrow {\rm{Hom_{\mathscr{C}}}}(A,C),\ \ (f',g')\mapsto f'\circ g' \\
  &{\rm{Hom}}_{\mathscr{D}}(A,B) \times {\rm{Hom_{\mathscr{C}}}}(B,C) \rightarrow {\rm{Hom}}_{\mathscr{D}}(A,C),\ \ (f',g)\mapsto f'\circ g  \\
  &{\rm{Hom_{\mathscr{C}}}}(A,B) \times {\rm{Hom}}_{\mathscr{D}}(B,C) \rightarrow {\rm{Hom}}_{\mathscr{D}}(A,C),\ \ (f,g')\mapsto f\circ g'
  \end{align*}
  \item [(2)]
  For every $A,B\in {\rm{Ob}}(\mathscr{C})$ and the identity morphism $1'_{A}\in {\rm{Hom}}_{\mathscr{D}}(A,A)$, we have $f\circ 1'_{A}\in {\rm{Hom}}_{\mathscr{D}'}(A,B)$ for any morphism $f\in {\rm{Hom_{\mathscr{D}}}}(A,B)$. And for any $g\in {\rm{Hom_{\mathscr{D}}}}(B,A)$, we have $1'_{A}\circ f\in {\rm{Hom}}_{\mathscr{D}'}(B,A)$.
  \item[(3)]
  The compositions defined above are associative.
\end{itemize}
Then we obtain a preadditive factorization category $\overline{\mathscr{C}}=(\mathscr{C},\Sigma)=\underline{(\mathscr{C},\mathscr{D})}$ with ${\rm{Ob}}(\overline{\mathscr{C}})={\rm{Ob}}(\mathscr{C})$, ${\rm{Hom_{\overline{\mathscr{C}}}}}(A,B)={\rm{Hom}}_{\mathscr{C}}(A,B)$, and $An(A,B)={\rm{Hom}}_{\mathscr{D}}(A,B)$.
\end{construction}

\begin{definition}
Let $\mathscr{C}$ be a preadditive category. The preadditive factorization category $\overline{\mathscr{C}}$ constructed above is called the \textit{preadditive factorization category associated to} $\mathscr{C}$. And the pair $(\mathscr{C},\mathscr{D})$ of two equivalent preadditive categories as above is called the \textit{generator of the associated preadditive factorization category}.
\end{definition}

If $\mathscr{C}$ is a factorization category with products, we want to study products in $\mathscr{C}$. We will prove two anti-analogs of the universal property of a product (cf. \cite{Hilton}, Chapter II, 5) in a factorization category. Such properties are called the \textit{anti-universal properties} of the product.

\begin{proposition}[Anti-universal properties]\label{A28}\
\begin{itemize}
  \item [(1)]
  Let $(X_{i})_{i\in I}$ be a family of objects of $\mathscr{C}$ indexed by a set $I$ and let $X=\prod_{i\in I}X_{i}$ be the product of $X_{i}$ with projections $p_{i}:X\rightarrow X_{i}$. Then given any object $Y\in{\rm{Ob}}(\mathscr{C})$ and anti-morphism $f_{i}^{*}:Y\rightarrow X_{i}$ in $\mathscr{C}$, there exists a unique anti-morphism $f^{*}:Y\rightarrow X$ such that $f_{i}^{*}=p_{i}\circ f^{*}$, i.e. the following diagram commutes:
$$
\xymatrix{
  X  \ar[r]^{p_{i}} & X_{i}  \\
  Y \ar@{-->}[u]^{f^{*}}\ar[ur]_{f_{i}^{*}} &    }
$$
  \item [(2)]
  Let $(X_{i})_{i\in I}$ be a family of objects of $\mathscr{C}$ indexed by a set $I$ and let $X=\prod_{i\in I}X_{i}$ be the product of $X_{i}$ with anti-projections $p_{i}^{*}:X\rightarrow X_{i}$ ($p_{i}^{*}=p_{i}\circ1^{*}$). Then given any object $Y\in{\rm{Ob}}(\mathscr{C})$ and anti-morphism $f_{i}^{*}:Y\rightarrow X_{i}$ in $\mathscr{C}$, there exists a unique morphism $f:Y\rightarrow X$ such that $f_{i}^{*}=p_{i}^{*}\circ f$, i.e. the following diagram commutes:
$$
\xymatrix{
  X  \ar[r]^{p_{i}^{*}} & X_{i}  \\
  Y \ar@{-->}[u]^{f}\ar[ur]_{f_{i}^{*}} &    }
$$
\end{itemize}

\end{proposition}
\begin{proof}
The universal property of the product yields the following commutative diagram:
$$
\xymatrix{
  X  \ar[r]^{p_{i}} & X_{i}  \\
  Y \ar@{-->}[u]^{\exists!f}\ar[ur]_{f_{i}} &    }
$$
i.e. we have $f_{i}=p_{i}f$ for $i\in I$.

Then the following calculations give us the desired results:
\begin{align*}
f_{i}\circ1^{*}&=p_{i}\circ f\circ1^{*} \\
f_{i}^{*}&=p_{i}\circ f^{*},
\end{align*}
and
\begin{align*}
1^{*}\circ f_{i}&=1^{*}\circ p_{i}\circ f \\
f_{i}^{*}&=p_{i}^{*}\circ f.
\end{align*}
\end{proof}

\begin{definition}
Let $(X_{i})_{i\in I}$ be a family of objects of $\mathscr{C}$ indexed by a set $I$ and let $X=\prod_{i\in I}X_{i}$ be the product of $X_{i}$ with projections $p_{i}:X\rightarrow X_{i}$. Then \textit{the corresponding anti-product} of $X$ is the pair $(X;p_{i}^{*})$ with anti-projections $p_{i}^{*}=p_{i}\circ1^{*}$ such that the anti-universal property described in Proposition \ref{A28}-(2) above is satisfied.
\end{definition}

In the following, we will prove two anti-analogues of \cite{Hilton}, Chapter II, Theorem 5.1, using anti-universal properties developed above.

\begin{theorem}
If $(X;p_{i})$ and $(X';p_{i}')$ are both products of the family $(X_{i})_{i\in I}$ in $\mathscr{C}$, then there exists a unique anti-isomorphism $g^{*}:X\xrightarrow{\sim}X'$ such that $p_{i}^{*}=p_{i}'\circ g^{*}$.
\end{theorem}
\begin{proof}
The anti-universal property described in Proposition \ref{A28}-(2) yields the following commutative diagrams:
$$
\xymatrix{
  X  \ar[r]^{p_{i}} & X_{i}  \\
  X' \ar@{-->}[u]^{f^{*}}\ar[ur]_{p_{i}'^{*}} &    }
$$
i.e. $p_{i}'^{*}=p_{i}\circ f^{*}$; and
$$
\xymatrix{
  X'  \ar[r]^{p_{i}'} & X_{i}  \\
  X \ar@{-->}[u]^{g^{*}}\ar[ur]_{p_{i}^{*}} &    }
$$
i.e. $p_{i}^{*}=p_{i}'\circ g^{*}$. Then we have $p_{i}'^{*}=p_{i}'^{*}\circ g^{*}\circ f^{*}$ and $p_{i}^{*}=p_{i}^{*}\circ f^{*}\circ g^{*}$, which imply that $g^{*}\circ f^{*}=1$ and $f^{*}\circ g^{*}=1$. Hence $g^{*}$ is the unique anti-isomorphism as desired.
\end{proof}

\begin{theorem}
If $(X;p_{i}^{*})$ and $(X';p_{i}'^{*})$ are both anti-products of the family $(X_{i})_{i\in I}$ in $\mathscr{C}$, then there exists a unique isomorphism $g:X\xrightarrow{\sim}X'$ such that $p_{i}^{*}=p_{i}'^{*}\circ g$.
\end{theorem}
\begin{proof}
The anti-universal property described in Proposition \ref{A28}-(1) yields the following commutative diagrams:
$$
\xymatrix{
  X  \ar[r]^{p_{i}^{*}} & X_{i}  \\
  X' \ar@{-->}[u]^{f}\ar[ur]_{p_{i}'^{*}} &    }
$$
i.e. $p_{i}'^{*}=p_{i}^{*}\circ f$; and
$$
\xymatrix{
  X'  \ar[r]^{p_{i}'^{*}} & X_{i}  \\
  X \ar@{-->}[u]^{g}\ar[ur]_{p_{i}^{*}} &    }
$$
i.e. $p_{i}^{*}=p_{i}'^{*}\circ g$. Then we have $p_{i}'^{*}=p_{i}'^{*}\circ g\circ f$ and $p_{i}^{*}=p_{i}^{*}\circ f\circ g$, which imply that $g\circ f=1$ and $f\circ g=1$. Hence $g$ is the unique isomorphism as desired.
\end{proof}

\section{Factorable Functors}\label{B4}
In the following, we want to clarify the relation between two factorization categories. This motivates us to define functors between factorization categories.

\begin{lemma}\label{A35}
Let $\mathscr{C}$ and $\mathscr{D}$ be factorization categories and let $F:\mathscr{C}\rightarrow\mathscr{D}$ be a functor between the underlying categories. Then for any $x,y\in{\rm{Ob}}(\mathscr{C})$, the set map ${\rm{Hom_{\mathscr{C}}}}(x,y)\rightarrow {\rm{Hom_{\mathscr{D}}}}(Fx,Fy)$ induces a set map $An_{\mathscr{C}}(x,y)\rightarrow An_{\mathscr{D}}(Fx,Fy)$. Moreover, if $\mathscr{C}$ and $\mathscr{D}$ are preadditive factorization categories and $F$ is an additive functor of the underlying preadditive categories, then the group homomorphism ${\rm{Hom_{\mathscr{C}}}}(x,y)\rightarrow {\rm{Hom_{\mathscr{D}}}}(Fx,Fy)$ induces a group homomorphism $An_{\mathscr{C}}(x,y)\rightarrow An_{\mathscr{D}}(Fx,Fy)$.
\end{lemma}
\begin{proof}
We just need to prove the second statement. Proposition \ref{A17} yields the following diagram
$$
\xymatrix{
  \textrm{Hom}_{\mathscr{C}}(x,y) \ar@{=}[d]_{} \ar[r]^{} & \textrm{Hom}_{\mathscr{D}}(Fx,Fy) \ar@{=}[d]^{} \\
  An_{\mathscr{C}}(x,y)  & An_{\mathscr{D}}(Fx,Fy)   }
$$
Then the composition is our desired homomorphism.
\end{proof}

\begin{definition}\label{A36}
Let $(\mathscr{C},\Sigma_{1})$ and $(\mathscr{D},\Sigma_{2})$ be factorization categories. A \textit{factorable functor} between two factorization categories $(\mathscr{C},\Sigma_{1})$ and $(\mathscr{D},\Sigma_{2})$, denoted by $F:(\mathscr{C},\Sigma_{1})\rightarrow(\mathscr{D},\Sigma_{2})$, is given by the following data:
\begin{itemize}
\item[(1)]
A functor $F:\mathscr{C}\rightarrow\mathscr{D}$ between the underlying categories, which is called the \textit{underlying functor} of the factorable functor.
\item[(2)]
For every pair of objects $x,y\in\textrm{Ob}(\mathscr{C})$, the induced set map $\varphi_{x,y}:An_{\mathscr{C}}(x,y)\rightarrow An_{\mathscr{D}}(Fx,Fy)$ given in Lemma \ref{A35}.
\end{itemize}
These data are assumed to satisfy the following conditions:
\begin{itemize}
\item[(I)]
Given an anti-morphism $f^{*}\in An_{\mathscr{C}}(x,y)$, we have $F(f^{*})=\varphi_{x,y}(f^{*})\in An_{\mathscr{D}}(Fx,Fy)$.
\item[(II)]
Given morphisms $f,g$ in $\mathscr{C}$, when the compositions make sense, we have
$$
F(f^{*}g^{*})=F(f^{*})F(g^{*}),\ F(f^{*}g)=F(f^{*})F(g),\textrm{ and }F(fg^{*})=F(f)F(g^{*}).
$$
\end{itemize}
\end{definition}

\begin{definition}
Let $(\mathscr{C},\Sigma_{1})$ and $(\mathscr{D},\Sigma_{2})$ be preadditive factorization categories. An \textit{additive factorable functor} between two factorization categories $(\mathscr{C},\Sigma_{1})$ and $(\mathscr{D},\Sigma_{2})$, denoted by $F:(\mathscr{C},\Sigma_{1})\rightarrow(\mathscr{D},\Sigma_{2})$, is given by the following data:
\begin{itemize}
\item[(1)]
An additive functor $F:\mathscr{C}\rightarrow\mathscr{D}$ between the underlying abelian categories, which is called the \textit{underlying additive functor} of the factorable functor.
\item[(2)]
For every pair of objects $x,y\in\textrm{Ob}(\mathscr{C})$, the induced group homomorphism $\varphi_{x,y}:An_{\mathscr{C}}(x,y)\rightarrow An_{\mathscr{D}}(Fx,Fy)$ given in Lemma \ref{A35}.
\end{itemize}
These data are assumed to satisfy the conditions (I) and (II) in Definition \ref{A36}.
\end{definition}
\begin{remark}
Note that the condition (I) above is indeed superfluous. We mention it merely to clarify the definition. By abuse of notation, we simply denote a factorable functor by the notation of the underlying functor.
\end{remark}

The above definition readily yields the following proposition.
\begin{proposition}
Let $F:\mathscr{C}\rightarrow\mathscr{D}$ be a factorable functor between factorization categories with products. Then $F$ preserves anti-products of families of objects of $\mathscr{C}$.
\end{proposition}
\begin{proof}
By the definition of factorable functors, it is easy to check.
\end{proof}

Since we have the category $\textbf{Cat}$ of small categories and functors, this motivates us to define the category of small factorization categories.

\begin{definition}
A \textit{small factorization category} is a factorization category whose underlying abelian category is small.
\end{definition}

Then we can form the category of small factorization categories, denoted by $\textbf{Fa}$. The objects of $\textbf{Fa}$ are small factorization categories and the morphisms in $\textbf{Fa}$ are factorable functors.

We can define a functor from the category $\textbf{Fa}$ of small factorization categories to the category $\textbf{Cat}$ of small categories:
$$
\textbf{FCA}:\textbf{Fa}\rightarrow \textbf{Cat}
$$
which to every factorization category, it associates the underlying category: $(\mathscr{C},\Sigma) \mapsto\mathscr{C}$, forgetting the factorial structure. And for every factorable functor $F:(\mathscr{C},\Sigma_{1})\rightarrow(\mathscr{D},\Sigma_{2})$ of factorization categories, it associates the underlying functor $F^{\circ}:\mathscr{C}\rightarrow\mathscr{D}$ between the underlying categories, dropping the factorable property.

\begin{definition}
The functor $\textbf{FCA}:\textbf{Fa}\rightarrow \textbf{Cat}$ defined above is called the \textit{forgetfully factorial functor}.
\end{definition}

And also, we can define a functor from the category $\textbf{Cat}$ of small abelian categories to the category $\textbf{Fa}$ of small factorization categories:
$$
\textbf{CAF}:\textbf{Cat}\rightarrow \textbf{Fa}
$$
which assigns to each category its associated factorization category: $\mathscr{C}\mapsto\overline{\mathscr{C}}=(\mathscr{C},\Sigma)$, i.e. equipping it with a factorial structure. And to every functor $F:\mathscr{C}\rightarrow\mathscr{D}$ of categories, it associates a factorable functor $\overline{F}:\overline{\mathscr{C}}\rightarrow\overline{\mathscr{D}}$ between the associated factorization categories. The factorable functor $\overline{F}:\overline{\mathscr{C}}\rightarrow\overline{\mathscr{D}}$ is defined as follows.

Let $(\mathscr{C},\mathscr{C}')$ and $(\mathscr{D},\mathscr{D}')$ be generators of $\overline{\mathscr{C}}$ and $\overline{\mathscr{D}}$ respectively. Note that the equivalences of categories yield an induced functor $F':\mathscr{C}'\rightarrow\mathscr{D}'$. The factorable functor $\overline{F}$ is any factorable functor between factorization categories $\overline{\mathscr{C}}$ and $\overline{\mathscr{D}}$ that satisfies the following conditions:
\begin{itemize}
\item[(1)]
The underlying functor is given by the functor $F:\mathscr{C}\rightarrow\mathscr{D}$.
\item[(2)]
For any $x,y\in\textrm{Ob}(\mathscr{C})$ we are given the set map $\textrm{Hom}_{\mathscr{C}'}(x,y)\rightarrow\textrm{Hom}_{\mathscr{D}'}(Fx,Fy)$ which is induced by the induced functor $F':\mathscr{C}'\rightarrow\mathscr{D}'$.
\end{itemize}

It is easy to see that such a factorable functor $\overline{F}$ is unique and independent of the choice of generators.

\begin{definition}
The functor $\textbf{CAF}:\textbf{Cat}\rightarrow \textbf{Fa}$ defined above is called the \textit{factorization functor}.
\end{definition}

Moreover, Theorem \ref{A22} indicates that we can define a functor from the category $\textbf{Fa}$ of small factorization categories to the category $\textbf{Cat}$ of small categories:
$$
\textbf{Fa}\rightarrow \textbf{Cat}
$$
which assigns to each factorization category $\mathscr{C}$ its associated category $\widetilde{\mathscr{C}}$. And to each factorable functor $F:\mathscr{C}\rightarrow\mathscr{D}$, it assigns a functor $\widetilde{F}:\widetilde{\mathscr{C}}\rightarrow\widetilde{\mathscr{D}}$ between the associated categories. One can construct the functor $\widetilde{F}$ as follows:
\begin{itemize}
\item[(1)]
$\widetilde{F}A:=FA$ for every $A\in \textrm{Ob}(\mathscr{C})=\textrm{Ob}(\widetilde{\mathscr{C}})$.
\item[(2)]
If $f:A\rightarrow B$ is any morphism in $\mathscr{C}$, we let $\widetilde{F}f:=Ff\in \textrm{Hom}(\widetilde{F}A,\widetilde{F}B)$. And if $g^{*}:A\rightarrow B$ is any anti-morphism in $\mathscr{C}$, we let $\widetilde{F}g:=Fg^{*}\in \textrm{Hom}(\widetilde{F}A,\widetilde{F}B)$.
\end{itemize}
\begin{definition}
The functor $\textbf{Fa}\rightarrow \textbf{Cat}$ defined above is called the \textit{categorization functor}.
\end{definition}

Recall that given a ring $R$, we have the \textit{free functor} $Fr:\textbf{Sets}\rightarrow \textbf{Mod}_{R}$, which associates with every set $S$ the free $R$-module with $S$ as a basis. And we have the \textit{underlying functor} also known as the forgetful functor $U:\textbf{Mod}_{R}\rightarrow\textbf{Sets}$, which associates with every $R$-module its underlying set. Moreover, $Fr$ is left adjoint to $U$. In other words, for $S\in\textrm{Ob}(\textbf{Sets})$ and $A\in\textrm{Ob}(\textbf{Mod}_{R})$, we have a functorial bijection
$$
\textrm{Hom}_{\textbf{Mod}_{R}}(Fr(S),A)\xrightarrow{\sim}\textrm{Hom}_{\textbf{Sets}}(S,U(A)).
$$
This inspires us that the factorization functor $\textbf{CAF}$ may be left adjoint to the forgetfully factorable functor $\textbf{FCA}$. However, the following theorems will show that $\textbf{CAF}$ is both left and right adjoint to $\textbf{FCA}$.
\begin{lemma}
Let $\mathscr{C}\in{\rm{Ob}}(\textbf{Fa})$ and $\mathscr{D}\in{\rm{Ob}}(\textbf{Cat})$. Then we have $\textbf{CAF}(\textbf{FCA}(\mathscr{C}))=\mathscr{C}$ and $\textbf{FCA}(\textbf{CAF}(\mathscr{D}))=\mathscr{D}$. Moreover, if $f$ is a morphism in $\textbf{Fa}$, then we have $\textbf{CAF}(\textbf{FCA}(f))=f$; and if $g$ is a morphism in $\textbf{Cat}$, then we have $\textbf{FCA}(\textbf{CAF}(g))=g$.
\end{lemma}
\begin{proof}
The first equality is due to Axiom \ref{A29}. And the rest are obvious.
\end{proof}

\begin{theorem}\label{C1}
The forgetfully factorable functor $\textbf{FCA}$ is left adjoint to the factorization functor $\textbf{CAF}$. In other words, for $\mathscr{C}\in{\rm{Ob}}(\textbf{Fa})$ and $\mathscr{D}\in{\rm{Ob}}(\textbf{Cat})$, there is a functorial bijection
$$
\xi_{\mathscr{C},\mathscr{D}}:{\rm{Hom_{\textbf{Cat}}}}(\textbf{FCA}(\mathscr{C}),\mathscr{D})\xrightarrow{\sim}{\rm{Hom_{\textbf{Fa}}}}(\mathscr{C},\textbf{CAF}(\mathscr{D})).
$$
\end{theorem}
\begin{proof}
We define a function
$$
{\rm{Hom_{\textbf{Cat}}}}(\textbf{FCA}(\mathscr{C}),\mathscr{D})\rightarrow{\rm{Hom_{\textbf{Fa}}}}(\mathscr{C},\textbf{CAF}(\mathscr{D})),\ f\mapsto \textbf{CAF}(f)\circ g
$$
where $g:\mathscr{C}\rightarrow\textbf{CAF}(\textbf{FCA}(\mathscr{C}))$ is an identity morphism in \textbf{Fa}. Then we define
$$
{\rm{Hom_{\textbf{Fa}}}}(\mathscr{C},\textbf{CAF}(\mathscr{D}))\rightarrow{\rm{Hom_{\textbf{Cat}}}}(\textbf{FCA}(\mathscr{C}),\mathscr{D}),\ f'\mapsto g'\circ\textbf{FCA}(f')
$$
where $g':\textbf{FCA}(\textbf{CAF}(\mathscr{D}))\rightarrow\mathscr{D}$ is an identity morphism in \textbf{Cat}. Then it can be checked that we have defined a bijection
$$
{\rm{Hom_{\textbf{Cat}}}}(\textbf{FCA}(\mathscr{C}),\mathscr{D})\xrightarrow{\sim}{\rm{Hom_{\textbf{Fa}}}}(\mathscr{C},\textbf{CAF}(\mathscr{D})).
$$
Finally, if $\textbf{FCA}(B)\xrightarrow{f}A\xrightarrow{g_{1}}A'$ are morphisms in $\textbf{Cat}$ and $B'\xrightarrow{h}B$ is a morphism in $\textbf{Fa}$, then we have
\begin{align*}
\xi_{A',B'}(g_{1}\circ f\circ \textbf{FCA}(h))&=\textbf{CAF}(g_{1}\circ f\circ \textbf{FCA}(h))\circ g=\textbf{CAF}(g_{1})\circ \textbf{CAF}(f)\circ h\circ g, \\
\textbf{CAF}(g_{1})\circ\xi_{A,B}(f)\circ h&=\textbf{CAF}(g_{1})\circ \textbf{CAF}(f)\circ g\circ h, \\
\xi_{A',B'}(g_{1}\circ f\circ \textbf{FCA}(h))&=\textbf{CAF}(g_{1})\circ\xi_{A,B}(f)\circ h.
\end{align*}
So the bijection $\xi_{\mathscr{C},\mathscr{D}}$ is natural in both $\mathscr{C}$ and $\mathscr{D}$.
\end{proof}

\begin{theorem}\label{C2}
The factorization functor $\textbf{CAF}$ is left adjoint to the forgetfully factorable functor $\textbf{FCA}$. In other words, for $\mathscr{C}\in{\rm{Ob}}(\textbf{Cat})$ and $\mathscr{D}\in{\rm{Ob}}(\textbf{Fa})$, there is a functorial bijection
$$
\xi_{\mathscr{C},\mathscr{D}}:{\rm{Hom_{\textbf{Fa}}}}(\textbf{CAF}(\mathscr{C}),\mathscr{D})\xrightarrow{\sim}{\rm{Hom_{\textbf{Cat}}}}(\mathscr{C},\textbf{FCA}(\mathscr{D})).
$$
\end{theorem}
\begin{proof}
We define a function
$$
{\rm{Hom_{\textbf{Fa}}}}(\textbf{CAF}(\mathscr{C}),\mathscr{D})\rightarrow{\rm{Hom_{\textbf{Cat}}}}(\mathscr{C},\textbf{FCA}(\mathscr{D})), \ \ f\mapsto \textbf{FCA}(f),
$$
whose inverse is the function
$$
{\rm{Hom_{\textbf{Cat}}}}(\mathscr{C},\textbf{FCA}(\mathscr{D}))\rightarrow {\rm{Hom_{\textbf{Fa}}}}(\textbf{CAF}(\mathscr{C}),\mathscr{D}),\ \ g\mapsto \textbf{CAF}(g).
$$
Thus, we get a bijection
$$
{\rm{Hom_{\textbf{Fa}}}}(\textbf{CAF}(\mathscr{C}),\mathscr{D})\xrightarrow{\sim}{\rm{Hom_{\textbf{Cat}}}}(\mathscr{C},\textbf{FCA}(\mathscr{D})).
$$

If $\textbf{CAF}(B)\xrightarrow{f}A\xrightarrow{g}A'$ are morphisms in $\textbf{Fa}$ and $B'\xrightarrow{h}B$ is a morphism in $\textbf{Cat}$, then we have
$$
\xi_{A',B'}(g\circ f\circ \textbf{CAF}(h))=\textbf{FCA}(g\circ f\circ \textbf{CAF}(h))=\textbf{FCA}(g)\circ \textbf{FCA}(f)\circ h,
$$
and
$$
\textbf{FCA}(g)\circ\xi_{A,B}(f)\circ h=\textbf{FCA}(g)\circ \textbf{FCA}(f)\circ h,
$$
which show that
$$
\xi_{A',B'}(g_{1}\circ f\circ \textbf{CAF}(h))=\textbf{FCA}(g)\circ\xi_{A,B}(f)\circ h.
$$
Hence the bijection $\xi_{\mathscr{C},\mathscr{D}}$ is natural in both $\mathscr{C}$ and $\mathscr{D}$.
\end{proof}

Next, we would like to study the subcategory of $\textbf{Fa}$. We use $\textbf{Fa-Pa}$ to indicate the subcategory of $\textbf{Fa}$ whose objects are small preadditive factorization categories and morphisms are additive factorable functors. And we denote by $\textbf{Cat-Pa}$ the subcategory of $\textbf{Cat}$ of small preadditive categories and additive functors.

We can define a functor from the category $\textbf{Fa-Pa}$ of small preadditive factorization categories to the category $\textbf{Cat-Pa}$ of small preadditive categories:
$$
\textbf{FCA-Pa}:\textbf{Fa-Pa}\rightarrow \textbf{Cat-Pa}
$$
which to every preadditive factorization category, it associates the underlying preadditive category: $(\mathscr{C},\Sigma) \mapsto\mathscr{C}$, forgetting the factorial structure. And for every additive factorable functor $F:(\mathscr{C},\Sigma_{1})\rightarrow(\mathscr{D},\Sigma_{2})$ of factorization categories, it associates the underlying additive functor $F^{\circ}:\mathscr{C}\rightarrow\mathscr{D}$ between the underlying preadditive categories, dropping the factorable property.

\begin{definition}
The functor $\textbf{FCA-Pa}:\textbf{Fa-Pa}\rightarrow \textbf{Cat-Pa}$ defined above is called the \textit{forgetfully preadditive factorial functor}.
\end{definition}

And also, we can define a functor from the category $\textbf{Cat-Pa}$ of small preadditive categories to the category $\textbf{Fa}$ of small preadditive factorization categories:
$$
\textbf{CAF-Pa}:\textbf{Cat-Pa}\rightarrow \textbf{Fa-Pa}
$$
which assigns to each preadditive category its associated preadditive factorization category: $\mathscr{C}\mapsto\overline{\mathscr{C}}=(\mathscr{C},\Sigma)$, i.e. equipping it with a factorial structure. And to every additive functor $F:\mathscr{C}\rightarrow\mathscr{D}$ of categories, it associates an additive factorable functor $\overline{F}:\overline{\mathscr{C}}\rightarrow\overline{\mathscr{D}}$ between the associated preadditive factorization categories. The additive factorable functor $\overline{F}:\overline{\mathscr{C}}\rightarrow\overline{\mathscr{D}}$ is defined as follows.

Let $(\mathscr{C},\mathscr{C}')$ and $(\mathscr{D},\mathscr{D}')$ be generators of $\overline{\mathscr{C}}$ and $\overline{\mathscr{D}}$ respectively. Note that the equivalences of categories yield an induced additive functor $F':\mathscr{C}'\rightarrow\mathscr{D}'$. The additive factorable functor $\overline{F}$ is any additive factorable functor between preadditive factorization categories $\overline{\mathscr{C}}$ and $\overline{\mathscr{D}}$ that satisfies the following conditions:
\begin{itemize}
\item[(1)]
The underlying additive functor is given by the additive functor $F:\mathscr{C}\rightarrow\mathscr{D}$.
\item[(2)]
For any $x,y\in\textrm{Ob}(\mathscr{C})$ we are given the group homomorphism $\textrm{Hom}_{\mathscr{C}'}(x,y)\rightarrow\textrm{Hom}_{\mathscr{D}'}(Fx,Fy)$ which is induced by the induced additive functor $F':\mathscr{C}'\rightarrow\mathscr{D}'$.
\end{itemize}

It is easy to see that such an additive factorable functor $\overline{F}$ is unique and independent of the choice of generators.

\begin{definition}
The functor $\textbf{CAF-Pa}:\textbf{Cat-Pa}\rightarrow \textbf{Fa-Pa}$ defined above is called the \textit{factorization preadditive functor}.
\end{definition}

Moreover, Theorem \ref{A22} indicates that we can define a functor from the category $\textbf{Fa-Pa}$ of small preadditive factorization categories to the category $\textbf{Cat}$ of small categories:
$$
\textbf{Fa-Pa}\rightarrow \textbf{Cat}
$$
which assigns to each preadditive factorization category $\mathscr{C}$ its associated category $\widetilde{\mathscr{C}}$. And to each factorable additive functor $F:\mathscr{C}\rightarrow\mathscr{D}$, it assigns a functor $\widetilde{F}:\widetilde{\mathscr{C}}\rightarrow\widetilde{\mathscr{D}}$ between the associated categories. One can construct the functor $\widetilde{F}$ as follows:
\begin{itemize}
\item[(1)]
$\widetilde{F}A:=FA$ for every $A\in \textrm{Ob}(\mathscr{C})=\textrm{Ob}(\widetilde{\mathscr{C}})$.
\item[(2)]
If $f:A\rightarrow B$ is any morphism in $\mathscr{C}$, we let $\widetilde{F}f:=Ff\in \textrm{Hom}(\widetilde{F}A,\widetilde{F}B)$. And if $g^{*}:A\rightarrow B$ is any anti-morphism in $\mathscr{C}$, we let $\widetilde{F}g:=Fg^{*}\in \textrm{Hom}(\widetilde{F}A,\widetilde{F}B)$.
\end{itemize}
\begin{definition}
The functor $\textbf{Fa-Pa}\rightarrow \textbf{Cat}$ defined above is called the \textit{categorization preadditive functor}.
\end{definition}

We can also show that the factorization preadditive functor $\textbf{CAF-Pa}$ is both left and right adjoint to the forgetfully preadditive factorial functor $\textbf{FCA-Pa}$. The proof are almost the same as Theorem \ref{C1} and \ref{C2}.

\begin{theorem}
The forgetfully preadditive factorial functor $\textbf{FCA-Pa}$ is left adjoint to the factorization preadditive functor $\textbf{CAF-Pa}$. In other words, for $\mathscr{C}\in{\rm{Ob}}(\textbf{Fa-Pa})$ and $\mathscr{D}\in{\rm{Ob}}(\textbf{Cat-Pa})$, there is a functorial bijection
$$
\xi_{\mathscr{C},\mathscr{D}}:{\rm{Hom_{\textbf{Cat-Pa}}}}(\textbf{FCA-Pa}(\mathscr{C}),\mathscr{D})\xrightarrow{\sim}{\rm{Hom_{\textbf{Fa-Pa}}}}(\mathscr{C},\textbf{CAF-Pa}(\mathscr{D})).
$$
\end{theorem}

\begin{theorem}
The factorization preadditive functor $\textbf{CAF-Pa}$ is left adjoint to the forgetfully preadditive factorial functor $\textbf{FCA-Pa}$. In other words, for $\mathscr{C}\in{\rm{Ob}}(\textbf{Cat-Pa})$ and $\mathscr{D}\in{\rm{Ob}}(\textbf{Fa-Pa})$, there is a functorial bijection
$$
\xi_{\mathscr{C},\mathscr{D}}:{\rm{Hom_{\textbf{Fa-Pa}}}}(\textbf{CAF-Pa}(\mathscr{C}),\mathscr{D})\xrightarrow{\sim}{\rm{Hom_{\textbf{Cat-Pa}}}}(\mathscr{C},\textbf{FCA-Pa}(\mathscr{D})).
$$
\end{theorem}

In fact, if we consider the subcategory $\textbf{Fa-Ab}$ of $\textbf{Fa}$ of small abelian factorization categories and factorable additive functors, we can deduce the same results in a similar way. Let $\textbf{Cat-Ab}$ be the subcategory of $\textbf{Cat}$ of small abelian categories and additive functors. We use $\textbf{CAF-Ab}$ to indicate the restriction of $\textbf{CAF-Pa}$ to $\textbf{Cat-Ab}$. And we denote by $\textbf{FCA-Ab}$ the restriction of $\textbf{FCA-Pa}$ to $\textbf{Fa-Ab}$. Then we have the following theorems:
\begin{theorem}
The functor $\textbf{FCA-Ab}$ is left adjoint to the functor $\textbf{CAF-Ab}$. In other words, for $\mathscr{C}\in{\rm{Ob}}(\textbf{Fa-Ab})$ and $\mathscr{D}\in{\rm{Ob}}(\textbf{Cat-Ab})$, there is a functorial bijection
$$
\xi_{\mathscr{C},\mathscr{D}}:{\rm{Hom_{\textbf{Cat-Ab}}}}(\textbf{FCA-Ab}(\mathscr{C}),\mathscr{D})\xrightarrow{\sim}{\rm{Hom_{\textbf{Fa-Ab}}}}(\mathscr{C},\textbf{CAF-Ab}(\mathscr{D})).
$$
\end{theorem}

\begin{theorem}
The functor $\textbf{CAF-Ab}$ is left adjoint to the functor $\textbf{FCA-Ab}$. In other words, for $\mathscr{C}\in{\rm{Ob}}(\textbf{Cat-Ab})$ and $\mathscr{D}\in{\rm{Ob}}(\textbf{Fa-Ab})$, there is a functorial bijection
$$
\xi_{\mathscr{C},\mathscr{D}}:{\rm{Hom_{\textbf{Fa-Ab}}}}(\textbf{CAF-Ab}(\mathscr{C}),\mathscr{D})\xrightarrow{\sim}{\rm{Hom_{\textbf{Cat-Ab}}}}(\mathscr{C},\textbf{FCA-Ab}(\mathscr{D})).
$$
\end{theorem}


\begin{thebibliography}{99}
\bibitem{Dr. Dale T. B.}
Dr. Dale T. B., \emph{Anti-Homomorphism in Rings}, International Journal of All Research Education and Scientific Methods (IJARESM), ISSN: 2455-6211
Volume 9, Issue 1, January-2021.
\bibitem{Gopalakrishnamoorthy}
Chandrasekhara Rao, K. and G.Gopalakrishnamoorthy, \emph{Anti-Homomorphisms in Groups}, Jour. of Inst. of Math. Comp.Sci. (2006, 2009, 2010).
\bibitem{Chandrasekhara}
K. Chandrasekhara Rao and Swaminathan Venkataraman, \emph{Anti-homomorphisms in near-rings}, Journal of Institute of Mathematics and Computer Sciences. Mathematics Series. 21., January 2008.
\bibitem{Pursell}
Pursell, Lyle E. , \emph{Anti-Isomorphisms vs. Isomorphisms}, Mathematics Magazine, vol. 44, no. 2, Mathematical Association of America, 1971, pp. 102-03, \url{https://doi.org/10.2307/2688925}.
\bibitem{Grillet}
Pierre Antoine Grillet, \emph{Abstract Algebra}, 2nd., Springer-Verlag New York, 2007.
\bibitem{Hilton}
P.J. Hilton and Urs Stammbach, \emph{A course in homological algebra}, 2nd., Springer-Verlag New York, 1997.
\bibitem{Gabber1}
Ofer Gabber and Lorenzo Ramero, \emph{Foundations For Almost Ring Theory}, Release 7.5, 2019.
\bibitem{Stack Project}
Stack project authors, \emph{Stack Project}, \url{https://stacks.math.columbia.edu/}, 2021.
\bibitem{Poonen}
Bjorn Poonen, \emph{Why all rings should have a 1}, April 1, 2014, arXiv:1404.0135v1.
\clearpage
\end{thebibliography}
 \end{document}